\documentclass[11pt]{article}
\usepackage{amsmath,amssymb,amsthm,amscd}
\usepackage[all]{xy}

\usepackage{graphics}

\newtheorem{theorem}{Theorem}[section]
\newtheorem{lemma}[theorem]{Lemma}
\newtheorem{proposition}[theorem]{Proposition}
\newtheorem{corollary}[theorem]{Corollary}
\newtheorem{example}[theorem]{Example}

\theoremstyle{plain}

\theoremstyle{definition}
\newtheorem{definition}[theorem]{Definition}
\newtheorem{remark}[theorem]{Remark}

\numberwithin{equation}{section}

\renewcommand{\theenumi}{(\roman{enumi})}
\renewcommand{\labelenumi}{\textup{(\theenumi)}}

\title{
A class of Exel--Laca algebras reciprocal to Cuntz--Krieger algebras} 
\author{Kengo Matsumoto \\
Department of Mathematics \\
Joetsu University of Education \\
Joetsu, Niigata 943-8512, Japan,
\and
Taro Sogabe \\
Faculty of Advanced Science and Technology \\
Kumamoto University \\
Kurokami 2-40-1, Chuo-ku, Kumamoto 860-8555, Japan
}

\begin{document}

%\date{2024, Feb 19}

\maketitle

\date{}

\def\det{{{\operatorname{det}}}}

%\maketitle
\begin{abstract}
The reciprocality means a duality in Kirchberg algebras   
between  $\operatorname{K}$-theory groups and strong extension groups.
In the paper, we will find a certain class of unital simple Exel--Laca algebras
for which the reciprocal duals are simple Cuntz--Krieger algebras
in terms of the underlying infinite matrices.
In our  procedure to obtain simple Cuntz--Krieger algebras from Exel--Laca algebras,
we compute the strong extension groups for Exel--Laca algebras belonging to the class. 
 \end{abstract}



%\tableofcontents

\newcommand{\Ker}{\operatorname{Ker}}
\newcommand{\sgn}{\operatorname{sgn}}
\newcommand{\Ad}{\operatorname{Ad}}
\newcommand{\ad}{\operatorname{ad}}
\newcommand{\orb}{\operatorname{orb}}
\newcommand{\rank}{\operatorname{rank}}

\def\Re{{\operatorname{Re}}}
\def\det{{{\operatorname{det}}}}
\newcommand{\K}{\operatorname{K}}

\newcommand{\sqK}{\operatorname{K}\!\operatorname{K}}

\newcommand{\bbK}{\mathbb{K}}
\newcommand{\N}{\mathbb{N}}
\newcommand{\bbC}{\mathbb{C}}
\newcommand{\R}{\mathbb{R}}
\newcommand{\Rp}{{\mathbb{R}}^*_+}
\newcommand{\T}{\mathcal{T}}
\newcommand{\bbT}{\mathbb{T}}

\newcommand{\Z}{\mathbb{Z}}
\newcommand{\Zp}{{\mathbb{Z}}_+}
\def\AF{{{\operatorname{AF}}}}

\def\TorZ{{{\operatorname{Tor}}^\Z_1}}
\def\Ext{{{\operatorname{Ext}}}}
\def\Exts{\operatorname{Ext}_{\operatorname{s}}}
\def\Extw{\operatorname{Ext}_{\operatorname{w}}}
\def\Ext{{{\operatorname{Ext}}}}
\def\Free{{{\operatorname{Free}}}}

\def\Ks{\operatorname{K}^{\operatorname{s}}}
\def\Kw{\operatorname{K}^{\operatorname{w}}}

\def\OA{{{\mathcal{O}}_A}}
\def\ON{{{\mathcal{O}}_N}}
\def\OAT{{{\mathcal{O}}_{A^t}}}
\def\OAI{{{\mathcal{O}}_{A^\infty}}}
\def\OAIn{{{\mathcal{O}}_{A^\infty_n}}}
\def\OAInone{{{\mathcal{O}}_{A^\infty_{n+1}}}}

\def\OalgAI{{{\mathcal{O}}^{\operatorname{alg}}_{A^\infty}}}

\def\OAIN{{{\mathcal{O}}_{A^\infty_N}}}
\def\OATI{{{\mathcal{O}}_{A^{t \infty}}}}

\def\PAI{{{\mathcal{P}}_{A^\infty}}}
\def\FAI{{{\mathcal{F}}_{A^\infty}}}
\def\IAI{{{\mathcal{I}}_{A^\infty}}}
\def\IAIn{{{\mathcal{I}}_{A^\infty}^n}}

\def\OSA{{{\mathcal{O}}_{S_A}}}

\def\TA{{{\mathcal{T}}_A}}
\def\TAn{{{\mathcal{T}}_{A_n}}}
\def\TAnone{{{\mathcal{T}}_{A_{n+1}}}}
\def\TAN{{{\mathcal{T}}_{A_N}}}

\def\whatA{{\widehat{\A}}}
\def\whatB{{\widehat{\B}}}

\def\TN{{{\mathcal{T}}_N}}

\def\TAT{{{\mathcal{T}}_{A^t}}}

\def\TB{{{\mathcal{T}}_B}}
\def\TBT{{{\mathcal{T}}_{B^t}}}

\def\A{{\mathcal{A}}}
\def\B{{\mathcal{B}}}
\def\C{{\mathcal{C}}}
\def\D{{\mathcal{D}}}
\def\O{{\mathcal{O}}}
\def\OaA{{{\mathcal{O}}^a_A}}
\def\OB{{{\mathcal{O}}_B}}
\def\OTA{{{\mathcal{O}}_{\tilde{A}}}}
\def\F{{\mathcal{F}}}
\def\G{{\mathcal{G}}}
\def\FA{{{\mathcal{F}}_A}}
\def\PA{{{\mathcal{P}}_A}}
\def\PI{{{\mathcal{P}}_\infty}}
\def\OI{{{\mathcal{O}}_\infty}}

\def\OalgI{{{\mathcal{O}}^{\operatorname{alg}}_\infty}}

\def\calI{\mathcal{I}}
\def\calK{\mathcal{K}}
\def\calP{\mathcal{P}}
\def\calQ{\mathcal{Q}}
\def\calR{\mathcal{R}}
\def\calC{\mathcal{C}}
\def\calD{\mathcal{D}}
\def\calM{\mathcal{M}}

\def\bbC{{\mathbb{C}}}

\def\U{{\mathcal{U}}}
\def\OF{{{\mathcal{O}}_F}}
\def\DF{{{\mathcal{D}}_F}}
\def\FB{{{\mathcal{F}}_B}}
\def\DA{{{\mathcal{D}}_A}}
\def\DB{{{\mathcal{D}}_B}}
\def\DZ{{{\mathcal{D}}_Z}}

\def\End{{{\operatorname{End}}}}

\def\Ext{{{\operatorname{Ext}}}}
\def\Hom{{{\operatorname{Hom}}}}

\def\Tor{{{\operatorname{Tor}}}}

\def\Max{{{\operatorname{Max}}}}
\def\Max{{{\operatorname{Max}}}}
\def\max{{{\operatorname{max}}}}
\def\KMS{{{\operatorname{KMS}}}}
\def\Per{{{\operatorname{Per}}}}
\def\Out{{{\operatorname{Out}}}}
\def\Aut{{{\operatorname{Aut}}}}
\def\Ad{{{\operatorname{Ad}}}}
\def\Inn{{{\operatorname{Inn}}}}
\def\Int{{{\operatorname{Int}}}}
\def\det{{{\operatorname{det}}}}
\def\exp{{{\operatorname{exp}}}}
\def\nep{{{\operatorname{nep}}}}
\def\sgn{{{\operatorname{sign}}}}
\def\cobdy{{{\operatorname{cobdy}}}}
\def\Ker{{{\operatorname{Ker}}}}
\def\Coker{{{\operatorname{Coker}}}}
\def\Im{{\operatorname{Im}}}

\def\ind{{{\operatorname{ind}}}}
\def\Ind{{{\operatorname{Ind}}}}
\def\id{{{\operatorname{id}}}}
\def\supp{{{\operatorname{supp}}}}
\def\co{{{\operatorname{co}}}}
\def\scoe{{{\operatorname{scoe}}}}
\def\coe{{{\operatorname{coe}}}}
\def\I{{\mathcal{I}}}
\def\Span{{{\operatorname{Span}}}}
\def\event{{{\operatorname{event}}}}
\def\Proj{{{\operatorname{Proj}}}}
\def\S{\mathcal{S}}

\def\PI{\mathcal{P}_{\infty}}
\def\PAI{{{\mathcal{P}}_{A^\infty}}}

\def\whatOA{\widehat{\mathcal{O}}_A}
\def\whatOAT{\widehat{\mathcal{O}}_{A^t}}

\def\sAI{\sigma_{A^\infty}}
\def\swhatA{\sigma_{\widehat{A}}}
\def\OalgAI{{{\mathcal{O}}^{\operatorname{alg}}_{A^\infty}}}

\def\wtO{\widetilde{O}}
\def\wtA{\widetilde{A}}
\def\wtAI{\widetilde{A}_{\infty}}
\def\wttAI{\widetilde{A}^t_{\infty}}
\def\E{\mathcal{E}}
\def\opE{\operatorname{E}}
\def\CR{\color{red}}
\def\CB{\color{blue}}
\def\gATI{{\gamma_{A^{t \infty}}}}

\def\OCKA{{\O}^{\text{\tiny CK}}_A}
\def\OELA{{\O}^{\text{\tiny EL}}_A}
\def\OELwA{{\O}^{\text{\tiny EL}}_{\widehat{A}}}

\def\OCK{{\O}^{\text{\tiny CK}}}
\def\OEL{{\O}^{\text{\tiny EL}}}

\def\wtT{\widetilde{\mathcal{T}}}
\def\wtTAn{\widetilde{\mathcal{T}}_{A_n}}
\def\wtTAN{\widetilde{\mathcal{T}}_{A_N}}

\def\Cokern{\Coker(I_n - \widetilde{A}_{C_n})}
\def\Cokernplus1{\Coker(I_{n+1} - \widetilde{A}_{C_{n+1}})}

%\newpage

%%%%%%%%%%%%%%%%%%%%%%%%%%%%%%%%%%%%%%%

%%%%%%%%%%%%%%%%%%%%%%%%%%%%%%%%%%%%%%%%%%%%%%%%%%%%%%%%%%%%%%
\section{Introduction}
%%%%%%%%%%%%%%%%%%%%%%%%%%%%%
The Spanier--Whitehead $\K$-dual, introduced in \cite{KS}, of a Kirchberg algebra $\A$ 
with finitely generated $\K$-groups
is defined as a Kirchberg algebra $D(\A)$ satisfying 
$\K_*(\A)\cong \K^*(D(\A))$, where $\K^*(\,\,\,)$ is the $\K$-homology groups.
The Kirchberg algebra $D(\A)$ is uniquely determined by $\A$ up to $\sqK$-equivalence.
Recently, the second named author  introduced in \cite{SogabeJFA}
another kind of $\K$-theoretic duality called the reciprocality.     
The notion of reciprocality in Kirchberg algebras with finitely generated $\K$-groups 
was introduced %by the second named author in \cite{SogabeJFA}
to investigate the homotopy groups of the automorphism groups and bundles of Kirchberg algebras.
Two Kirchberg algebras $\A, \B$ with finitely generated $\K$-groups
are said to be {\it reciprocal}\/ if 
$\A \sim_{\sqK} D(C_\B)$ and 
$\B \sim_{\sqK} D(C_\A)$, where $D(C_\A)$ and $D(C_\B)$ are the Spanier--Whitehead 
$\K$-duals of the mapping cone algebras $C_\A$
%$(=\{ x \in C_{(0,1]}\otimes\A \mid x(1) \in \mathbb{C}\})$ 
and $C_\B$, respectively, 
and $\sim_{\sqK}$ 
is a $\sqK$-equivalence.
If $\A$ and $\B$ are reciprocal, then $\A$ (resp. $\B$) is said to be the reciprocal 
dual of $\B$ (resp. $\A$).
The reciprocal dual of $\A$ is uniquely determined by $\A$ up to isomorphism,
and written as $\widehat{\A}$. 
As in \cite{MatSogabe2} and \cite{MatSogabe3}, 
the reciprocal duality is a duality between the $\K$-groups 
$\K_*(\,\,\,)$ and the strong extension groups $\Exts^*(\,\,\,)$,
  whereas the Spanier-Whitehead $\K$-duality is a duality between the $\K$-groups 
$\K_*(\,\,\,)$ and the weak extension groups $\Extw^*(\,\,\,)$.
So, two Kirchberg algebras $\A$ and $\B$ are reciprocal if and only if
$\K_i(\A) \cong \Exts^{i+1}(\B)$ and  
$\K_i(\B) \cong \Exts^{i+1}(\A)$ for $i=0,1$ hold (cf. \cite[Proposition 3.7]{MatSogabe2}). 
For an irreducible non-permutation matrix $A$ with entries in $\{0,1\}$,
if the size of the matrix is finite, 
the associated Cuntz--Krieger algebra 
is denoted by $\OCKA$ (\cite{CK}), and
if the size of the matrix is countably infinite,
 the associated unital Exel--Laca algebra 
is denoted by $\OELA$(\cite{EL}).
Both of the classes provide important examples of Kirchberg algebras.

In the series of papers \cite{MatSogabe}, \cite{MatSogabe2}, 
\cite{MatSogabe3}, \cite{MatSogabe4} and \cite{MatSogabe5},
the authors have been studying the reciprocal duality in Kirchberg algebras.
It was proved in \cite{MatSogabe} that  the homotopy groups of the automorphism
groups of Cuntz--Krieger algebras are complete invariants of the isomorphism classes
of the algebras  by using the reciprocality.   
As an application, we showed that the two groups $\Extw(\OCKA)$ and $\Exts(\OCKA)$
are a complete set of invariants of the isomorphism class of $\OCKA$.
The discussions were generalized to a wider class of Kirchberg algebras 
with finitely generated $\K$-groups in \cite{MatSogabe2}.  
One of the most interesting and important problems was to find
and realize the reciprocal dual $\widehat{\OCKA}$ of a simple Cuntz--Krieger algebra $\OCKA$
in a concrete way in terms of the original matrix $A$. 
We then found the answer of the problem in \cite{MatSogabe3} in the following way. 
Let $A=[A(i,j)]_{i,j=1}^N$ be an $N \times N$ irreducible non-permutation 
finite matrix with entries
in $\{0,1\}$. 
The reciprocal dual $\widehat{\OCKA}$ of the simple Cuntz--Krieger algebra $\OCKA$
is realized to be the unital Exel--Laca algebra $\OELwA$ 
defined by the infinite irreducible matrix $\widehat{A}$ such that     
\begin{equation}\label{eq:widehatA}
\widehat{A}=
\left[
\begin{array}{ccccccc}
&                &1      &0     &0     &\cdots \\
&\text{\Huge A}^t&\vdots&\vdots&\vdots&\vdots \\
&                &1      &0     &0     &\cdots \\
1&\dots          &1      &1     &1     &\cdots \\
1&\dots          &1      &0     &0     &\cdots \\
1&\dots          &1      &0     &0     &\cdots \\
\vdots &\dots    &\vdots &\vdots&\vdots&\vdots \\
\end{array}
\right].
\end{equation}
We have to mention that 
the unital Exel--Laca algebra $\OELwA$ is a Kirchberg algebra such that 
$\rank(\K_0(\OELwA)) - \rank(\K_1(\OELwA)) =1$, whereas   
the original Cuntz--Krieger algebra $\OCKA$ satisfies
$\rank(\K_0(\OCKA)) - \rank(\K_1(\OCKA)) =0$
(\cite[Proposition 4.4]{MatSogabe3}).
Hence the reciprocal dual $\OELwA$ can not be realized as a simple Cuntz--Krieger algebra.

In this paper, we will  find a certain class of unital Exel--Laca algebras containing 
the class of algebras $\OEL_{\widehat{A}}$ for the matrices $\widehat{A}$ defined in \eqref{eq:widehatA}
 in terms of the underlying matrices
for which the reciprocal dual is realized as a simple Cuntz--Krieger algebra.
 We will then clarify the matrix operations to obtain the finite irreducible matrices
 defining the Cuntz--Krieger algebras
  from the primary given infinite matrices defining the Exel--Laca algebras.
 This gives us  the reverse procedure of the operation 
 from simple Cuntz--Krieger algebras $\OCKA$ to the reciprocal 
 dual $\widehat{\OCKA} ( = \OELwA)$.    
The class which we will forcus on in the Kirchberg algebras  
is the unital Exel--Laca algebras 
whose underlying matrices satisfy the following properties called (RSF), (RS), (DRS) and (LI).
An infinite matrix $A=[A(i,j)]_{i,j\in \N}$ naturally associates to an infinite  directed graph 
$G_A =(V_A, E_A)$ with vertex set $V_A = \N$.  
A vertex $i \in V_A$ is said to be row finite (resp. row complementary finite)
if the set $\{j\in \N \mid A(i,j) =1\}$  
(resp. $\{j\in \N \mid A(i,j) =0\}$) is finite, in this case we write $c_i =0$ (resp. $c_i =1$).
The matrix $A$ is said to be row semi-finite written (RSF) for short 
 if every vertex is either row finite or row complementary finite,
and there is at least one row complementary finite vertex.
The matrix $A$ satisfying (RSF)
 is said to be  right stable written (RS) 
 if there exists $K \in \N$ such that for any $n \ge K$,
the equalities $A(i,n+1) = c_i$ hold for $i=1,\dots, n$. 
If the additional diagonal condition, called (DC),
$A(n+1,n+1) = c_{n+1}$ for $n \ge K$ is satisfied, 
then it is said to be diagonally right stable written (DRS) for short.
The matrix $A$ is said to be locally irreducible written (LI) for short
 if the restricted $n \times n$ matrix $A_n =[A(i,j)]_{i,j=1}^n$
is irreducible for $n \ge K$.
The matrix $\widehat{A}$ defined in \eqref{eq:widehatA}
is (RSF), (DRS) and (LI) for every irreducible non-permutation  
finite matrix $A$. 
 The following theorem is our main result in this paper.

\begin{theorem}[{Theorem \ref{thm:main2}}]\label{thm:mainth2}
Let $A = [A(i,j)]_{i,j\in \N}$
be an irreducible infinite matrix with entries in $\{0,1\}$.
 Assume that $A$ is (DRS) and (LI).
Then there exists $K \in \N$ such that the reciprocal dual 
$\widehat{\OELA}$ of the unital Exel--Laca algebra $\OELA$
 is isomorphic to the simple Cuntz--Krieger algebra 
$\O^{\text{\tiny CK}}_{\widehat{A}_{C_K}}$ for the $(K+2)\times (K+2)$-finite matrix 
$\widehat{A}_{C_K}$ defined by 
\begin{equation}\label{eq:matrixhatACK}
\widehat{A}_{C_K}
:= 
\begin{bmatrix}
    &       &    & 1 & 0 \\
    & A_K^t &    & \vdots & \vdots \\
    &       &    & 1 & 0 \\
   1& \dots & 1  & 1 & 1 \\
c_1 & \dots & c_K& 1 & 1 
\end{bmatrix},
\quad
\end{equation}
where $A_K^t$ is the transposed matrix of the $K\times K$ matrix
$A_K = [A(i,j)]_{i,j=1}^K$ defined by the restriction of $A$.
\end{theorem}
 Since
$\A\cong\widehat{\widehat{\A}}(=\widehat{(\widehat{\A})})$,
the above theorem shows that, for $K=N+1$ and the infinite matrix 
$\widehat{A}$ defined by \eqref{eq:widehatA} 
with a finite irreducible non-permutation matrix $A=[A(i,j)]_{i,j=1}^N$,
we have 
\begin{equation}\label{eq:doubledualA}
 \OCKA \cong \widehat{\OELwA} \cong \OCK_{\widehat{\widehat{A}}_{C_K}}
\end{equation}
and 
\begin{equation}\label{eq:tildeAhatN}
\widehat{\widehat{A}}_{C_K} 
(:= \widehat{(\widehat{A})}_{C_K})
=
\begin{bmatrix}
  &                &  & 1&  1      & 0 \\
  & A              &  & \vdots  & \vdots & \vdots \\ 
  &                &  & 1& \vdots      & \vdots \\
1 & \dots          & 1& 1& 1      & 0      \\
%0 & \dots          & 0& 1& 1      & 0 \\
1 & \dots          & 1& 1& 1      & 1 \\
0 & \dots          & 0& 1& 1      & 1 
\end{bmatrix}
\qquad
(\text{see } \, \, \eqref{eq:tildeBhatK}).
 \end{equation}

 %Let $A=[A(i,j)]_{i,j=1}^N$ be an $N \times N$ irreducible non-permutation matrix with entries in $\{0,1\}.$
As an application, 
%we may know that the pair $\Exts^1(\widehat{\OCKA})$ and $\Extw^1(\widehat{\OCKA})$ of the two extension groups of $\OCKA$ is a computable complete set of isomorphism invariants of $\OCKA$
we simplify our previous result (\cite[Corollary 1.3.]{MatSogabe2}) in the following way. 
\begin{proposition}[{Corollary \ref{cor:InvariantforCK}}] \label{prop:mainprop}
For an $N \times N$ irreducible non-permutation matrix $A$ with entries in $\{0,1\},$
the group
\begin{equation}
(\Z^N/(I_N - A^t)\Z^N)\oplus(\Z^N/
\begin{bmatrix}
&     & 1 \\
& I_N - A^t & \vdots \\
&     & 1
\end{bmatrix} \Z^{N+1})
\end{equation}
is a complete invariant of the isomorphism class of the simple Cuntz--Krieger algebra $\OCKA$.
\end{proposition}

To prove  Theorem \ref{thm:mainth2}, 
we need to prove that the Exel--Laca algebra $\OELA$ defined by an infinite matrix satisfying 
(DRS) and (LI) has finitely generated $\K$-groups,
and  find a concrete Kirchberg algebra $\A$ such that
\begin{equation}\label{eq:ExtsK}
(\Exts^1(\OELA), [\iota_{\OELA}(1)]_s, \Exts^0(\OELA)
 \cong 
 (\K_0(\A), [1_\A]_0, \K_1(\A)),
\end{equation}
 where 
  $\iota_{\OELA}(1)$ is an extension of $\OELA$  satisfying 
  $\Exts^1(\OELA)/\Z [\iota_{\OELA}(1)]_s \cong \Extw^1(\OELA).$
We  prove that the Kirchberg algebra $\A$ may be taken as a simple Cuntz--Krieger algebra 
$\OCK_{\widehat{A}_{C_K}}.$
The main ingredient of this paper  is to find the underlying matrix $\widehat{A}_{C_K}$
as in \eqref{eq:matrixhatACK}
which defines the Cuntz--Krieger algebra
in terms of the given infinite matrix.

In Section \ref{sect:inductivelimit}, 
we will prove that $\OELA$ may be written as an
inductive limit of Toeplitz like algebras $\wtT_{A_n}$ (Lemma \ref{lem:EL1}).
In Section \ref{sect:Kinductivelimit}, 
we will present a formula to compute the 
$\K$-theory groups $\K_*(\OELA)$, so that we know that   
$\K_*(\OELA)$ is finitely generated  if the infinite matrix $A$ is (DRS) and (LI) (Proposition \ref{prop:KformulaEL}).
Since the first strong extension group $\Exts^1(\OELA)$ is isomorphic to
the strong extension group $\Exts(\OELA)$ defined by the strong equivalence classes
of Busby invariants $\tau:\OELA\rightarrow \calQ(H)$ which are unital
 $*$-monomorphisms from $\OELA$ to the Calkin algebra $\calQ(H)$,  
 the group $\Exts(\OELA)$ is computed following the strategy in \cite{MaAnalMath2024}.
 We first compute $\Exts(\wtT_{A_n})$,
 and then we see that 
 $\Exts(\OELA)$ is written as a projective limt $\varprojlim \Exts(\wtT_{A_n})$.
 A formula for the extension groups $\Ext_*(\OELA)$ is given in Proposition \ref{prop:mainExtsExtw1}. 
 The key ingredient to compute the left hand side of of the triplet \eqref{eq:ExtsK}
 is to compute 
the cyclic six term exact sequence 
\begin{equation}\label{eq:cyclicsixtermforExtOA}
\begin{CD}
\Exts^0(\OELA) @>>> \Extw^0(\OELA) @>>>  \Z \\
@AAA @. @V{\iota_{\OELA}}VV \\
0 @<<<  \Extw^1(\OELA) @<<< \Exts^1(\OELA)
\end{CD}
\end{equation}
in terms of the underlying infinite matrix $A=[A(i,j)]_{i,j\in \N}$
(Theorem \ref{thm:6termforExtOA}),
which will be fullfiled in  Section \ref{sect:cyclicsixtermforExt}. 
By Theorem \ref{thm:6termforExtOA},
we may find a finite matrix $\widetilde{A}_{C_n}$
 having its entries in possibly nagative integers such that 
$\Exts^1(\OELA) \cong \Coker(I_{n+2} - \widetilde{A}_{C_n})$
and
$\Exts^0(\OELA) \cong \Ker(I_{n+2} - \widetilde{A}_{C_n}).$
We will then replace the matrix $\widetilde{A}_{C_n}$
to $A_{C_n}$ having its entries in $\{0,1\}$
and replace it again with $\widehat{A}_{C_n}$ 
to adjust the position in the $\K_0$-group
of the unit of the Cuntz--Krieger  
$\OCK_{\widehat{A}_{C_n}}$.
We finally obtain
a simple Cuntz--Krieger algebra $\A$ satisfying 
\eqref{eq:ExtsK} in a concrete way and show Theorem \ref{thm:main2} (Theorem \ref{thm:mainth2}).
In Section \ref{sect:Examples}, we will find the matrix 
$\widehat{\widehat{A}}_{C_K}$ for a given finite irreducible matrix $A=[A(i,j)]_{i,j=1}^N$
satisfying \eqref{eq:doubledualA}.
If the given matrix $A=[A(i,j)]_{i,j\in \N}$ does not (DRS),
the reciprocal dual of $\OELA$ does not necessarily  realize a simple Cuntz--Krieger algebra.
Such an example is the matrix $P_\infty$ defined in
\eqref{eq:Pinftymatrix} whose Exel--Laca algebra $\OEL_{P_\infty}$
is the Kirchberg algebra $\PI$ 
satisfying $\K_0(\PI) =0, \K_1(\PI) = \Z$ (Section \ref{sect:Examples}).

By the reciprocality, we know that the pair
$(\Exts^1(\widehat{\OCKA}), \Extw^1(\widehat{\OCKA}))$
of the extension groups is a complete set of the isomorphism invariant
of the Cuntz--Krieger algebra $\OCKA$.
By describing the pair in terms of the underlying matrix $A$,
we have a computable complete invariant of the simple Cuntz--Krieger algebra $\OCKA$
(Proposition \ref{prop:InvariantforCK}, Corollary \ref{cor:InvariantforCK}).

Throughout the paper, Kirchberg algebras mean separable unital purely infinite simple nuclear 
$C^*$-algebras satisfying the UCT. 
The set of positive integers is denoted by $\N$.
The identity matrix of size $n$ is denoted by $I_n$.

%%%%%%%%%%%%%%%%%%%%%%%%%%%%%%%%%%%%%%%%%%%%%%%%%%%%%%%%
\section{Preliminaries} \label{sec:Preliminaries}
%%%%%%%%%%%%%%%%%%%%%%%%%%%%%%%%%%%%%%%%%%%%%%%%%
%%%%%%%%%%%%%%%%%%%%%%%%%%%%%%%%%%%%%%%%%%

Let $A=[A(i,j)]_{i,j\in \N}$ be an infinite matrix with entries in $\{0,1\}$.
Throughout the paper, we assume that $A$ has no zero rows or columns.
The unital Exel--Laca algebra $\OELA$ defined by the infinite matrix $A$
was introduced and studied in \cite{EL}.
See \cite{EL}, \cite{ELKtheory} for detail.
The matrix associates to a directed graph $G_A=(V_A, E_A)$
whose vertex set $V_A$ 
is the set of positive integers $\N$ and edge set $E_A$ is the set of pairs $(i,j) \in \N \times \N$
such that $A(i,j) =1$.  
Let us 
define the complementary matrix $A^c=[A^c(i,j)]_{i,j\in \N}$ of $A$ by
$A^c(i,j) = 1 -A(i,j)$ for $i,j \in \N$.
\begin{definition}\label{def:rowfinite}
\begin{enumerate}
\renewcommand{\theenumi}{(\roman{enumi})}
\renewcommand{\labelenumi}{\textup{\theenumi}}
\item A vertex $i$ is said to be {\it row finite} if
 the set $\{ j \in \N \mid A(i,j) =1\}$ is finite.
\item A vertex $i$ is said to be {\it row complementary finite} if
 the set $\{ j \in \N \mid A^c(i,j) =1\}$ is empty or finite.
\end{enumerate}
\end{definition} 
The set of row finite (resp. row complementary finite)
vertices is denoted by $V_{RF}$ (resp. $V_{RCF}$). 
 Note that 
 $
V_{RF} \cap V_{RCF} = \emptyset
$
and 
\[ (V_{RF}\cup V_{RSF})^c=\{i\in\mathbb{N}\mid \{A(i, j)\}_{j\in \N} \;\text{contains infinitely many}\; 0 \;\text{and}\; 1\}.\]

For $i \in \N$, put
\begin{equation}\label{eq:ci}
c_i:=
\begin{cases}
0 & \text{ if } i \text{ is row finite},\\
1 & \text{ if } i \text{ is row complementary finite.}
\end{cases}
\end{equation}

\begin{definition}\label{def:semifinite}
\begin{enumerate}
\renewcommand{\theenumi}{(\roman{enumi})}
\renewcommand{\labelenumi}{\textup{\theenumi}}
\item A matrix $A$ is  said to be {\it row semi-finite}\/ written as (RSF) for short if
 $V_A = V_{RF} \cup V_{RCF}$ and $V_{RCF} \ne \emptyset$.
\item 
A row semi-finite matrix  $A$ is  said to be {\it right stable}\/ 
written as (RS) for short
 if there exists $K \in \N$ such that
 for any $n \ge K$ the condition 
$A(i,n+1) = c_i$ holds for all $i=1,\dots,n.$
\item 
A right stable  matrix  $A$ is  said to be {\it diagonally right stable}\/ 
written as (DRS) for short
 if there exists $K \in \N$ such that
 for any $n \ge K$ the condition 
$A(i,n+1) = c_i$ holds for all $i=1,\dots,n+1.$
\end{enumerate}
\end{definition} 
The condition $A(n+1,n+1) = c_{n+1}$ for all $n \ge K$
is said to be {\it diagonal condition}\/ written as (DC) for short.
Hence (DRS) $\Longleftrightarrow$ (RS) $+$ (DC).
%%%%%%%%%%%%%%%%%%%%%%%%%%%%%%%%%%%%%%%%%%%%%%%%%%%%%%

If $A$ is (RS), then it is of the form
\begin{equation*}
\left[
\begin{array}{ccccccc}
A(1,1)  & \dots & A(1,K) &  c_1     &  c_1     & c_1  & \dots \\ 
A(2,1)  & \dots & A(2,K) &  c_2     &  c_2     & c_2  & \dots \\ 
\vdots  &       & \vdots &\vdots    &\vdots    &\vdots& \dots \\ 
A(K,1)  & \dots & A(K,K) & c_K      & c_K      & c_K  & \dots \\ 
A(K+1,1)& \dots &A(K+1,K)&A(K+1,K+1)&c_{K+1}   &c_{K+1}& \dots \\ 
A(K+2,1)& \dots &A(K+2,K)&A(K+2,K+1)&A(K+2,K+2)&c_{K+2}& \dots \\ 
\vdots &       & \vdots & \vdots    & \vdots   &\vdots  &\\ 
\end{array}
\right]
\end{equation*}
Furthermore if $A$ is (DRS), then the above diagonal entries
$A(n+1,n+1)$  are $c_{n+1}$ for $n \ge K$. 

The condition (RS) is defined under the condition (RSF).
Hence if the matrix $A$ is (RS), then it is (RSF),
and there exists $i \in \N$ such that $c_i=1$.
 One may then take the number $K$ in Definition \ref{def:semifinite} (ii)
 such that $V_{RCF} \cap \{1,\dots,K\} \ne \emptyset$,
 if we choose $K$ large enough.
%%%%%%%%%%%%%%%%%%%%%%
%\begin{remark}
%Assume that an infinite matrix $A=[A(i,j)]_{i,j\in \N}$ is (RSF).
%Let $S_i, i \in \N$ be the canonical generating partial isometries of 
%the  the Exel--Laca algebra $\OELA$ for the matrix $A$.
%Since $V_{RCF} \ne \emptyset$, the Exel--Laca algebra $\OELA$ has an unit.
%For a vertex $i \in \N$, we have
%\begin{enumerate}
%\renewcommand{\theenumi}{(\roman{enumi})}
%\renewcommand{\labelenumi}{\textup{\theenumi}}
%\item
%$i \in V_{RF}$ if and only if 
%there exists a finite subset $F_i\subset \N$
%such that $S_i^* S_i = \sum_{j \in F_i} S_j S_j^*.$
%\item
%$i \in V_{RCF}$ if and only if 
%there exists a finite subset $F^c_i\subset \N$
%such that $S_i^* S_i + \sum_{j \in F^c_i} A^c(i,j) S_j S_j^* = 1.$
%\end{enumerate}
%\end{remark}
In what follows,
the infinite matrix  
$A$ is assumed to be (RS)
for the number $K$ in Definition \ref{def:semifinite} (ii)
such that $c_i =1$ for some $i \in \N$ satisfying $1 \le i \le K$.
For $n \ge K$, the $n\times n$ submatrix $A_n$ of $A$ is defined by
\begin{equation*}
A_n(i,j) = A(i,j), \qquad i,j \in \{1,\dots,n\}. 
\end{equation*}
\begin{definition}
Keep the above situation.
The matrix $A$ is said to be {\it locally irreducible}\/ written as (LI) for short
if  the submatrices
$A_n$ are all irreducible for all $n\ge K$.
\end{definition}
%%%%%%%%%%%%%%%%%%%%%%%%%%%%%%%%%%%%%%%%%%%%%%%%%%%%%%

\begin{example}\label{ex:threetypes}

{\bf 1.}
Let $I_\infty$ be the infinite matrix whose entries are all $1'$s
such as 
\begin{equation}\label{eq:Iinftymatrix}
I_\infty =
\begin{bmatrix}
1 & 1 & 1 &\cdots &   & \\
1 & 1 & 1 &\cdots &   & \\
1 & 1 & 1 &\cdots &   & \\
\vdots  &\vdots &\vdots& &   & \\            
  &   &   &   &    &             
\end{bmatrix}.
\end{equation}
All the vertices are row complementary finite.
Hence the matrix $I_\infty$ is (DRS) and (LI)
with $K=1$.

{\bf 2.}
Let $A=[A(i,j)_{i,j=1}^N$ be an $N\times N$ 
irreducible non-permutation matrix with entries in $\{0,1\}$.
As in \cite{MatSogabe3}, the reciprocal dual 
$\widehat{\OCKA}$ of the Cuntz--Krieger algebra $\OCKA$ is 
a unital Kirchberg algebra 
defined by the Exel--Laca algebra 
$\OELwA$ 
for the  infinite matrix $\widehat{A}$
defined by \eqref{eq:widehatA}.
%%%%%%%%%%%%%%%%%
%\begin{equation}\label{eq:widehatA}
%\widehat{A}=
%\left[
%\begin{array}{ccccccc}
%&                &1      &0     &0     &\cdots \\
%&\text{\Huge A}^t&\vdots&\vdots&\vdots&\vdots \\
%&                &1      &0     &0     &\cdots \\
%1&\dots          &1      &1     &1     &\cdots \\
%1&\dots          &1      &0     &0     &\cdots \\
%1&\dots          &1      &0     &0     &\cdots \\
%\vdots &\dots    &\vdots &\vdots&\vdots&\vdots \\
%\end{array}
%\right]
%\end{equation}
The vertex $N+1\in V_{\widehat{A}}$ is row complementary finite, the other vertices
are all row finite. 
Hence the matrix $\widehat{A}$ is (DRS) and (IL)
with $K=N+1$.

{\bf 3.}
Let $P_\infty$ be the infinite matrix such as
\begin{equation}\label{eq:Pinftymatrix} 
P_\infty =
\begin{bmatrix}
1 & 0 & 1 & 1 & 1 & \cdots &   & \\
0 & 1 & 1 & 1 & 1 & \cdots &   & \\
1 & 0 & 1 & 0 & 0 & \cdots &   & \\
0 & 1 & 0 & 1 & 0 & \cdots &   & \\
0 & 0 & 1 & 0 & 1 &        &   & \\
  &\ddots &\ddots &\ddots &\ddots&\ddots &   & \\            
  &   &   &   &   &  &   &             
\end{bmatrix}
\quad (\text{cf. } \cite[Example \, 4.2]{RaebSzy})
\end{equation}
Both of the first and second vertices are row complementary finite, 
and the other vertices are all 
row finite. 
Hence $P_\infty$ is (RS) and (LI)
with $K=2$, but  not (DRS).
As in \cite[Example 4.2]{RaebSzy},
the Exel--Laca algebra $\O^{\text{\tiny EL}}_{P_\infty}$
is  the  unital Kirchberg algebra $\PI$ satisfying
$\K_0(\PI) =0, \K_1(\PI) =\Z$.

{\bf 4. }The matrix
$$
\begin{bmatrix}
 1 & 0 & 0 & 0 & 0 & \cdots \\
 0 & 1 & 0 & 0 & 0 & \cdots \\
 0 & 0 & 1 & 1 & 1 & \cdots \\
 0 & 0 & 1 & 1 & 1 & \cdots \\
 0 & 0 & 1 & 1 & 1 & \cdots \\
\vdots& \vdots& \vdots& \vdots & \vdots & \vdots & \ddots 
\end{bmatrix}
$$
is (DRS), but not (LI).
\end{example}
\begin{lemma}\label{lem:EL=>Kirchberg}
If $A$ is (RS) and (LI), then
the Exel--Laca algebra $\OELA$ is a Kirchberg algebra.
\end{lemma}
\begin{proof}
Since $A$ is (RSF), the vertex set $V_{RCF}$ is not empty so that
the $C^*$-algebra $\OELA$ is unital.
By (LI), it is easy to see that the matrix $A$ is irreducible and
for any $j\in \N$, there exists a path from $j$ to a row complementary vertex.
This shows that any loop has an exit, so the $C^*$-algebra $\OELA$ 
is purely infinite and simple by \cite{EL} (cf. \cite[Example 4.4.5]{RodSto}, \cite{Szy2001}, etc.).
\end{proof}

%%%%%%%%%%%%%%%%%%%%%%%%%%%%%%%%%%%%%%%%%%%%%%%%%%%%%%%%%%%%%%%%%%%%%%%
\section{$\OELA= \varinjlim \wtT_{A_n}$}\label{sect:inductivelimit}
%%%%%%%%%%%%%%%%%%%%%%%%%%%%%%%%%%%%%%%%%%%%%%%%%%%%%%%%%%%%%%%%%%%%%%%%%
In this section, we will express the Exel--Laca algebra 
$\OELA$ as an inductive limit $C^*$-algebra of a sequence.
$\wtT_{A_n}, n \ge K$
of Toeplitz like algebras.
\begin{lemma}\label{lem:EL1}
Assume that an infinite matrix $A=[A(i,j)]_{i,j\in \N}$
 is (RS) and (LI).
Let $S_i, i \in \N$ be the canonical generating partial isometries of 
 the Exel--Laca algebra $\OELA$ for the matrix $A$.
Let $K \in \N$ be the positive number in Definition \ref{def:semifinite} (ii).
Then for any $n \in \N$ with $n \ge K$
there exists a sequence of projections
$P_{n+1}, n\in \N$ in $\OELA$ satisfying
\begin{align}
\sum_{j=1}^n S_j S_j^* + P_{n+1} =1, & \qquad 
S_i^* S_i = \sum_{j=1}^n A(i,j) S_j S_j^* + c_i P_{n+1}, \quad i=1,\dots, n, \label{eq:EL11}\\
P_n = &  S_n S_n^* + P_{n+1}. \label{eq:Pn}
\end{align} 
\end{lemma}
\begin{proof}
Put $P_{n+1} = 1 -\sum_{j=1}^n S_j S_j^*$ which satisfies \eqref{eq:Pn}.
%Now the matrix $A$ is (WUT), so that one may take $N \in \N$ as in the definition of (WUT).
Assume $n \ge K$ and $i \in \{ 1,\dots, n\}$.

If $i \in V_{RF}$,
we have
$A(i,j) = c_i =0$ for $j \ge n+1$. 
Recall the following set and the universal relation (cf. \cite[eq. (1.3)]{EL}):
\[S(X, Y):=\{k\in\mathbb{N}\mid A(x, k)=1\;\textup{for all}\; x\in X,\; A(y, k)=0\;\textup{for all}\; y\in Y\},\]
\begin{equation} 
\left(\prod_{x\in X}S_x^*S_x\right)\left( \prod_{y\in Y}(1-S_yS_y^*)\right)=\sum_{k\in S(X, Y)}S_kS_k^*, \label{EL4}
\end{equation}
 for finite subsets $X, Y \subset \N.$
Note that our matrix $A$ provides a unital $\mathcal{O}_A^{EL}$ (see Lemma \ref{lem:EL=>Kirchberg}).
Applying the relation $\eqref{EL4}$ for $X=\{i\},\; Y=\emptyset$,
one has
\begin{equation*}
S_i^* S_i = \sum_{j=1}^n A(i,j) S_j S_j^*
= \sum_{j=1}^n A(i,j) S_j S_j^* + c_i P_{n+1}.
\end{equation*} 
If $i \in V_{RCF}$, 
we have 
$A(i,j) = c_i =1$ for $j \ge n+1$.
Applying the relation $\eqref{EL4}$ for $X=\emptyset, Y=\{i\}$,

\begin{equation*}
S_i^* S_i =  \sum_{j=1}^n A(i,j) S_j S_j^* + (1 - \sum_{j=1}^n S_j S_j^*) 
 = \sum_{j=1}^n A(i,j) S_j S_j^* + c_i P_{n+1}.
\end{equation*} 
\end{proof}

%%%%%%%%%%%%%%%%%%%%%%%%%%%%
We are assuming that a matrix $A$ is (RS) and (LI).
 Denote by $\wtTAn$
the unital $C^*$-subalgebra $C^*(S_1,\dots, S_n, P_{n+1})$ of $\OELA$ generated by 
$S_1, \dots, S_n $ and $P_{n+1}$.
 Note that $\sum_{i=1}^nS_i S_i^*+P_{n+1}=1_{\OELA}=1_{\tilde{\mathcal{T}}_{A_n}}.$
In the expression \eqref{eq:EL11}, 
the relation 
\begin{equation}\label{eq:R1}
S_i^* S_i = \sum_{j=1}^n A(i,j) S_j S_j^* + c_i P_{n+1} %, \qquad i=1,\dots,n
\end{equation}
holds for $1 \le i \le n$. On the other hand, we have
\begin{equation}\label{eq:R2}
S_i^* S_i = \sum_{j=1}^{n+1} A(i,j)  S_j S_j^* + c_i P_{n+2}. 
\end{equation}
Since $c_i = A(i,n+1)$ for $ 1 \le i \le n$,
 the right-hand side of \eqref{eq:R2} is
\begin{align*}
  \sum_{j=1}^{n}A(i,j)  S_j S_j^* + A(i, n+1) S_{n+1} S_{n+1}^*+ c_i P_{n+2} %\\
% = &  \sum_{j=1}^{n}A(i,j)  S_j S_j^* + c_i( S_{n+1} S_{n+1}^*+  P_{n+2}) \\
 =   \sum_{j=1}^{n} A(i,j) S_j S_j^* + c_i P_{n+1}
\end{align*}
which is the right hand side of \eqref{eq:R1}.
Hence, the relations \eqref{eq:EL11} are compatible  with the natural embedding
 $\wtT_{A_n} \hookrightarrow \wtT_{A_{n+1}}.$

Let us denote by $\calK(H)$ the $C^*$-algebra 
of compact operators on a separable infinite dimensional Hilbert space $H$. 
\begin{lemma}\label{lem:EL1}
Assume that a matrix $A$ is (RS) and (LI).
\begin{enumerate}
\renewcommand{\theenumi}{(\roman{enumi})}
\renewcommand{\labelenumi}{\textup{\theenumi}}
\item
There exists a natural unital increasing sequence, 
\begin{equation*}
\wtT_{A_n} \hookrightarrow \wtT_{A_{n+1}} \hookrightarrow \dots  
\hookrightarrow \OELA\quad (n\geq K)
\end{equation*}
of $C^*$-subalgebras of $\OELA$ such that 
the union $\cup_{n=K}^\infty \wtT_{A_n}$ is dense in $\OELA$.
Hence  the Exel--Laca algebra $\OELA$ is  the $C^*$-algebra of an inductive limit
$\varinjlim{\wtT_{A_n}}$.
\item 
Let $\OCK_{A_n}$ denote the Cuntz--Krieger algebra defined by the irreducible finite matrix $A_n$
for $n \ge K$.
There exists a short exact sequence:
\begin{equation*}
0 
\longrightarrow \calK(H) 
\longrightarrow \wtT_{A_n}
\longrightarrow \OCK_{A_n}
\longrightarrow 0
\end{equation*}
such that 
$\calK(H)$ is a unique essential ideal of $\wtT_{A_n}$.
\item
The $C^*$-subalgebra 
$\wtT_{A_n}$ for $n \ge K$
is the graph algebra $C^*(\widetilde{G}_{A_n})$
defined by a finite directed graph
$\widetilde{G}_{A_n}$  with a sink.
\end{enumerate}
\end{lemma}
\begin{proof}
Since (i) is clear, It suffices to prove (ii) and (iii).

(ii)
Take $n \ge K$. 
For a row complementary  finite vertex $i$ with $1 \le i \le K$ (that exists by the condition (RS)), 
we have $c_i =1$ so that  
 the relation \eqref{eq:EL11} shows that 
 $S_i P_{n+1}S_i^* \ne 0$.
 It is easy to see that the $C^*$-subalgebra of $\wtT_{A_n}$ generated by the projection
$P_{n+1}$ becomes an ideal of $\wtT_{A_n}$ and
is isomorphic to the $C^*$-algebra $\calK(H)$.
As $\OCK_{A_n}$ is simple, 
 the ideal generated by $P_{n+1}$ is the unique essential ideal of $\wtT_{A_n}$.

(iii)
Let us define a finite directed graph
$\widetilde{G}_{A_n} = (\widetilde{V}_{A_n}, \widetilde{E}_{A_n})$  with a sink
in the following way.
The vertex set $\widetilde{V}_{A_n}$ is denoted by 
$\{v_1, \dots, v_n, \tilde{v}_{n+1}\}$.
Define an edge $e_{(i,j)}$ for $i,j =1,\dots,n$ 
if $A(i,j) =1$ with source vertex $v_i$ and terget vertex $v_j$.
If $c_i =1$, then define an additional edge from $v_i$ to
$\tilde{v}_{n+1}$. The set of such edges is denoted by
$\widetilde{E}_{A_n}$, and the resulting finite directed graph  
$(\widetilde{V}_{A_n}, \widetilde{E}_{A_n})$ is denoted by $\widetilde{G}_{A_n}$.
Put
\begin{equation*}
\begin{cases}
s_{(i,j)}& := S_i S_j S_j^*, \quad \text{ if } A(i,j) =1  \text{ for }  i,j =1,\dots,n, \\
p_j & : = S_j S_j^*, \qquad \text{ for } j=1,\dots, n, \\
p_{n+1} & : = P_{n+1}, \\
s_{(i,n+1)} & : = S_i P_{n+1} \quad \, \, \,  \text{ if } c_i =1, \, \, i=1,\dots,n.
\end{cases}
\end{equation*}
It is easy to see that 
the $C^*$-algebra $C^*(s_{(i,j)}, p_j \mid i =1,\dots,n,\, \, j =1,\dots,n+1)$
generated by 
$s_{(i,j)}, p_j \,\,  i =1,\dots,n,\, j =1,\dots,n+1$
coincides with $\wtT_{A_n}$ and 
is the graph algebra $C^*(\widetilde{G}_{A_n})$
of the graph $\widetilde{G}_{A_n}$
with a sink $\tilde{v}_{n+1}$. 
\end{proof}
\begin{remark}
 The algebraic structure of the $C^*$-subalgebra $\wtT_{A_n}$ 
is not determined by only the matrix $A_n$.
The algebra $\wtT_{A_n}$ is determined by
the operator relations  \eqref{eq:EL11} which needs the information of the numbers
$c_i, i=1,\dots,n$ determined by the full entries of 
the matrix $A = [A(i,j)]_{i,j\in \N}$. 
\end{remark}

%%%%%%%%%%%%%%%%%%%%%%%%%%%%%%%%%%%%%%%%%%%%%%%%%%%%%%%%%%%%%%%%%%%%%%%%%%%%%%%%%%
\section{$\K_*(\OELA)= \varinjlim \K_*(\wtT_{A_n})$}\label{sect:Kinductivelimit}
%%%%%%%%%%%%%%%%%%%%%%%%%%%%%%%%%%%%%%%%%%%%%%%%%%%%%%%%%%%%%%%%%%%%%%%%%%%%%%%%%
Under the assumtion that the matrix $A$ is (RS) and (LI),
we take and fix $K\in \N$ in Definition \ref{def:semifinite} (ii)
such that $c_i =1$ for some $i\in \{1,\dots,K\}$.
For $n \ge K$, we define the $n\times 1$-matrix
\begin{equation}\label{eq:Cn}
C_n := 
\begin{bmatrix}
c_1\\
\vdots\\
c_n
\end{bmatrix}.
\end{equation} 
By using Raeburn--Szyma{\'n}sky \cite{RaebSzy} 
and
Drinen--Tomforde's $\K$-group formulas \cite{DrinenTom} for graph $C^*$-algebras of 
finite directed graphs with sinks (cf. \cite{Szy2002}),
we have the following lemma.
\begin{lemma}[{cf.  \cite[Theorem 4.1]{RaebSzy}, \cite[Theorem 3.1]{DrinenTom}}]\label{lem:EL1}
Assume that a matrix $A=[A(i,j)]_{i,j\in \N}$ is (RS) and (LI).
For $n \ge K$, there exist isomorphisms 
\begin{enumerate}
\renewcommand{\theenumi}{(\roman{enumi})}
\renewcommand{\labelenumi}{\textup{\theenumi}}
\item
$ \rho_n^0:\K_0(\wtT_{A_n}) \to \Coker(
\begin{bmatrix}
I_n - A_n^t\\
-C_n^t
\end{bmatrix}
: \Z^n \rightarrow \Z^{n+1})$
 such that $\rho_n^0([1_{\wtT_{A_n}}]) = [1_{n+1}],$
\item
$ \rho_n^1:\K_1(\wtT_{A_n}) \to \Ker(
\begin{bmatrix}
I_n - A_n^t\\
-C_n^t
\end{bmatrix}
: \Z^n \rightarrow \Z^{n+1}),$
\end{enumerate}
where
$
\begin{bmatrix}
I_n - A_n^t\\
-C_n^t
\end{bmatrix}
$ is the $(n+1)\times n$-matrix
defined by
\begin{equation}\label{eq:INAntCnt}
\begin{bmatrix}
I_n - A_n^t\\
-C_n^t
\end{bmatrix}
= 
{\begin{bmatrix}
1-A(1,1) & -A(2,1) & \dots & -A(n,1) \\
 -A(1,2) & 1-A(2,2) & \dots & -A(n,2) \\
\vdots   & \vdots  & \ddots  & \vdots  \\
 -A(1,n) & -A(2,n) & \dots & 1-A(n,n) \\
 -c_1    & -c_2    & \dots & -c_n
\end{bmatrix},} 
\end{equation}
and the isomorphism
$ \rho_n^0:\K_0(\wtT_{A_n}) \to \Coker(
\begin{bmatrix}
I_n - A_n^t\\
-C_n^t
\end{bmatrix}
: \Z^n \rightarrow \Z^{n+1})$
is given by the correspondence 
$[S_i S_i^*]_0 \rightarrow e_i$ for $i=1,\dots,n$ and
$[P_{n+1}]_0 \rightarrow e_{n+1}$,
where $e_i, i=1,\dots, n+1$ is the vector in $\Z^{n+1}$
whose $i$th entry is $1$ and the other entries are zeros,
 and $[1_{n+1}]$ denotes the class of the vector $1_{n+1} \in \Z^{n+1}$ 
all of whose entries are $1$s.
\end{lemma}
\begin{lemma}\label{lem:KEL2}
The map 
$
\begin{bmatrix}
y_1\\
\vdots \\
y_n\\
y_{n+1}
\end{bmatrix} \in \Z^{n+1}
\rightarrow
\begin{bmatrix}
y_1\\
\vdots \\
y_n\\
y_{n+1}\\
y_{n+1}
\end{bmatrix} \in \Z^{n+2}
$
induces a homomorphism
\begin{equation}\label{eq:varphin+1n}
\varphi_{n, n+1}:
\Coker(
\begin{bmatrix}
I_n - A_n^t\\
-C_n^t
\end{bmatrix}
: \Z^n \rightarrow \Z^{n+1})
\rightarrow
\Coker(
\begin{bmatrix}
I_{n+1} - A_{n+1}^t\\
-C_{n+1}^t
\end{bmatrix}
: \Z^{n+1} \rightarrow \Z^{n+2}).
\end{equation}
\end{lemma}
\begin{proof}
By (RS), the equalities
$A(i,n+1) = c_i$ for $i=1,\dots, n$ hold. 
We then have
\begin{align*}
 &\varphi_{n,n+1}(
\begin{bmatrix}
I_n - A_n^t\\
-C_n^t
\end{bmatrix}
\begin{bmatrix}
y_1\\
\vdots \\
y_n
\end{bmatrix}) \\
= &
{\begin{bmatrix}
1-A(1,1) & -A(2,1) & \dots & -A(n,1) \\
 -A(1,2) & 1-A(2,2) & \dots & -A(n,2) \\
\vdots   & \vdots  & \ddots & \vdots  \\
 -A(1,n) & -A(2,n) & \dots & 1-A(n,n) \\
 -c_1    & -c_2    & \dots & -c_n \\
 -c_1    & -c_2    & \dots & -c_n
 \end{bmatrix}
\begin{bmatrix}
y_1\\
\vdots \\
y_n
\end{bmatrix}
} \\
= &
{\begin{bmatrix}
1-A(1,1) & -A(2,1) & \dots & -A(n,1)& -A(n+1,1) \\
 -A(1,2) & 1-A(2,2) & \dots& -A(n,2)& -A(n+1,2) \\
\vdots   & \vdots  &\ddots & \vdots & \vdots \\
 -A(1,n) & -A(2,n) & \dots & 1-A(n,n)& -A(n+1,n) \\
 -A(1,n+1)& -A(2,n+1)& \dots & -A(n,n+1) & 1-A(n+1,n+1)\\
 -c_1    & -c_2    & \dots & -c_n & -c_{n+1}
 \end{bmatrix}
\begin{bmatrix}
y_1\\
\vdots \\
y_n \\
0
\end{bmatrix}
} \\
= &
\begin{bmatrix}
I_{n+1} - A_{n+1}^t\\
-C_{n+1}^t
\end{bmatrix}
\begin{bmatrix}
y_1\\
\vdots \\
y_n \\
0
\end{bmatrix}
\end{align*}
so that the map
$\varphi_{n,n+1}$
defined by \eqref{eq:varphin+1n}
is well-defined.
\end{proof}
The natural unital embedding 
$\wtT_{A_n}\hookrightarrow \wtT_{A_{n+1}}$
induces a homomorphism 
$\K_0(\wtT_{A_n})\rightarrow \K_0(\wtT_{A_{n+1}})$
written as $\iota_{n,n+1}^0$.
We identify 
$\K_0(\wtT_{A_n})$ with 
$\Coker(
\begin{bmatrix}
I_n - A_n^t\\
-C_n^t
\end{bmatrix}
: \Z^n \rightarrow \Z^{n+1}).
$
\begin{lemma}\label{lem:KEL3}
Assume that a matrix $A$ is (RS) and (LI).
\begin{enumerate}
\renewcommand{\theenumi}{(\roman{enumi})}
\renewcommand{\labelenumi}{\textup{\theenumi}}
\item The diagram

\begin{equation}\label{eq:CDK0wt}
\begin{CD}
\K_0(\wtT_{A_n})  @>{\iota_{n,n+1}^0}>> \K_0(\wtT_{A_{n+1}}) \\
 @V{\rho_n^0}VV  @VV{\rho_{n+1}^0}V \\
\Coker(
\begin{bmatrix}
I_n - A_n^t\\
-C_n^t
\end{bmatrix}
: \Z^n \rightarrow \Z^{n+1})
@>{\varphi_{n,n+1}}>>
\Coker(
\begin{bmatrix}
I_{n+1} - A_{n+1}^t\\
-C_{n+1}^t
\end{bmatrix}
: \Z^{n+1} \rightarrow \Z^{n+2}).
\end{CD}
\end{equation}

is commutative.
\item
Furthermore if $A$ satisfies (DC),  
the homomorphism $\varphi_{n, n+1}$ (see \eqref{eq:varphin+1n})  is an isomorphism, 
and hence  
$\iota_{n,n+1}^0: \K_0(\wtT_{A_n})  \rightarrow \K_0(\wtT_{A_{n+1}})$
is an isomorphism.
\end{enumerate}
\end{lemma}
\begin{proof}
(i) The $\K$-groups
$\K_0(\wtT_{A_n})$ and  
$\K_0(\wtT_{A_{n+1}})$
are generated by the classes of projections
$[S_1S_1^*]_0,\dots,[S_nS_n^*]_0,[P_{n+1}]_0$
and
$[S_1S_1^*]_0,\dots,[S_{n+1}S_{n+1}^*]_0,[P_{n+2}]_0$,
respectively.
By the identity
$P_{n+1} = 
S_{n+1} S_{n+1}^* + P_{n+2},
$
one knows that the diagram \eqref{eq:CDK0wt}
is commutative.

(ii) Further assume that $A$ satisfies (DC), so  the equalities
$A(i,n+1) = c_i$ for $i=1,\dots, n+1$ hold.
We then have
\begin{align*}
\begin{bmatrix}
I_{n+1} - A_{n+1}^t\\
-C_{n+1}^t
\end{bmatrix}
=
\begin{bmatrix}
1-A(1,1) & -A(2,1) & \dots & -A(n,1)& -A(n+1,1) \\
 -A(1,2) & 1-A(2,2) & \dots& -A(n,2)& -A(n+1,2) \\
\vdots   & \vdots  &\ddots & \vdots & \vdots \\
 -A(1,n) & -A(2,n) & \dots & 1-A(n,n)& -A(n+1,n) \\
 -A(1,n+1)& -A(2,n+1)& \dots & -A(n,n+1) & 1-A(n+1,n+1)\\
 -A(1,n+1)& -A(2,n+1)& \dots & -A(n,n+1) & -A(n+1,n+1)
 \end{bmatrix}.
\end{align*}
%We will first show the surjectivity of $\varphi_{n,n+1}$.
For any $[x_i]_{i=1}^{n+2} \in \Z^{n+2},$
by putting
\[z_i =0\;\text{for}\;\; i=1, \dots,n, \,\,
z_{n+1} = x_{n+1} - x_{n+2},\]
and  
\begin{gather*}
y_i = x_i + A(n+1, i) z_{n+1}\; \text{ for }\; i=1,\dots,n,\\
y_{n+1}=x_{n+2}+A(n+1, n+1)z_{n+1},
\end{gather*}
we have
\begin{equation*}
\begin{bmatrix}
x_1 \\
\vdots \\
x_n \\
x_{n+1} \\
x_{n+2}
\end{bmatrix}
-
\begin{bmatrix}
y_1 \\
\vdots \\
y_n \\
y_{n+1} \\
y_{n+1} 
\end{bmatrix}
=
\begin{bmatrix}
I_{n+1} - A_{n+1}^t\\
-C_{n+1}^t
\end{bmatrix}
\begin{bmatrix}
z_1 \\
\vdots \\
z_n \\
z_{n+1} 
\end{bmatrix}.
\end{equation*}  
This shows that the map
$\varphi_{n,n+1}$ is surjective.
We will next show the injectivity of the map $\varphi_{n,n+1}$. 
Suppose that $\varphi_{n,n+1}([y_i]_{i=1}^{n+1}) =0,$
and take 
 $[w_i]_{i=1}^{n+1} \in \Z^{n+1}$ such that 
\begin{equation*}
\begin{bmatrix}
y_1 \\
\vdots \\
y_n \\
y_{n+1} \\
y_{n+1} 
\end{bmatrix}
=
\begin{bmatrix}
I_{n+1} - A_{n+1}^t\\
-C_{n+1}^t
\end{bmatrix}
\begin{bmatrix}
w_1 \\
\vdots \\
w_n \\
w_{n+1} 
\end{bmatrix}
\end{equation*}  
By looking at the $(n+1)$th row and $(n+2)$th row,
one has $w_{n+1} =0,$
and the above equality implies
\begin{equation*}
\begin{bmatrix}
y_1 \\
\vdots \\
y_n \\
y_{n+1} 
\end{bmatrix}
%=
%\begin{bmatrix}
%1-A(1,1) & -A(2,1) & \dots & -A(n,1) \\
% -A(1,2) & 1-A(2,2) & \dots& -A(n,2) \\
%\vdots   & \vdots  &\ddots & \vdots  \\
% -A(1,n) & -A(2,n) & \dots & 1-A(n,n) \\
% -A(1,n+1)& -A(2,n+1) & \dots & -A(n,n+1)
% \end{bmatrix}
%\begin{bmatrix}
%w_1 \\
%\vdots \\
%w_n 
%\end{bmatrix}
%
=
\begin{bmatrix}
I_{n} - A_{n}^t\\
-C_{n}^t
\end{bmatrix}
\begin{bmatrix}
w_1 \\
\vdots \\
w_n 
\end{bmatrix},
\end{equation*}  
showing the injectivity of the map $\varphi_{n,n+1}$.
\end{proof}
\begin{lemma}\label{lem:KEL4}
The map 
$
\begin{bmatrix}
x_1\\
\vdots \\
x_n
\end{bmatrix} \in \Z^{n}
\rightarrow
\begin{bmatrix}
x_1\\
\vdots \\
x_n\\
0
\end{bmatrix} \in \Z^{n+1}
$
induces a homomorphism
\begin{equation}\label{eq:phin+1n}
\phi_{n,n+1}:
\Ker(
\begin{bmatrix}
I_n - A_n^t\\
-C_n^t
\end{bmatrix}
: \Z^n \rightarrow \Z^{n+1})
\rightarrow
\Ker(
\begin{bmatrix}
I_{n+1} - A_{n+1}^t\\
-C_{n+1}^t
\end{bmatrix}
: \Z^{n+1} \rightarrow \Z^{n+2}).
\end{equation}
\end{lemma}
\begin{proof}
By (RS), the equalities
$A(i,n+1) = c_i$ for $i=1,\dots, n$ hold. 
For 
$[x_i]_{i=1}^n \in \Z^n$ 
with
$
\begin{bmatrix}
I_{n} - A_{n}^t\\
-C_{n}^t
\end{bmatrix}
[x_i]_{i=1}^n
=0,$
we then have
\begin{align*}
 & {
 \begin{bmatrix}
I_{n+1} - A_{n+1}^t\\
-C_{n+1}^t
\end{bmatrix}
(\phi_{n,n+1}(
\begin{bmatrix}
x_1\\
\vdots \\
x_n
\end{bmatrix}) )
}\\
= &
{\begin{bmatrix}
1-A(1,1) & -A(2,1) & \dots & -A(n,1)& -A(n+1,1) \\
 -A(1,2) & 1-A(2,2) & \dots& -A(n,2)& -A(n+1,2) \\
\vdots   & \vdots  &\ddots & \vdots & \vdots \\
 -A(1,n) & -A(2,n) & \dots & 1-A(n,n)& -A(n+1,n) \\
 -A(1,n+1)& -A(2,n+1)& \dots & -A(n,n+1) & 1-A(n+1,n+1)\\
 -c_1    & -c_2    & \dots & -c_n & -c_{n+1}
 \end{bmatrix}
\begin{bmatrix}
x_1\\
\vdots \\
x_n \\
0
\end{bmatrix}
} \\
= &
{\begin{bmatrix}
1-A(1,1) & -A(2,1) & \dots & -A(n,1) \\
 -A(1,2) & 1-A(2,2) & \dots & -A(n,2) \\
\vdots   & \vdots  &\ddots & \vdots  \\
 -A(1,n) & -A(2,n) & \dots & 1-A(n,n) \\
 -c_1    & -c_2    & \dots & -c_n \\
 -c_1    & -c_2    & \dots & -c_n
 \end{bmatrix}
\begin{bmatrix}
x_1\\
\vdots \\
x_n
\end{bmatrix}
} = 
\begin{bmatrix}
0\\
\vdots \\
0 \\
0
\end{bmatrix}.
\end{align*}
so that the map
$\phi_{n,n+1}$
defined by \eqref{eq:phin+1n}
is well-defined.
\end{proof}
The natural unital embedding 
$\wtT_{A_n}\hookrightarrow \wtT_{A_{n+1}}$
induces a homomorphism 
$\K_1(\wtT_{A_n})\rightarrow \K_1(\wtT_{A_{n+1}})$
written as $\iota_{n,n+1}^1$.
\begin{lemma}\label{lem:KEL5}
Assume that a matrix $A$ is (RS) and (LI).
\begin{enumerate}
\renewcommand{\theenumi}{(\roman{enumi})}
\renewcommand{\labelenumi}{\textup{\theenumi}}
\item The diagram

\begin{equation}\label{eq:CDK1wt}
\begin{CD}
\K_1(\wtT_{A_n})  @>{\iota_{n,n+1}^1}>> \K_1(\wtT_{A_{n+1}}) \\
@V{\rho_n^1}VV  @VV{\rho_{n+1}^1}V \\
\Ker(
\begin{bmatrix}
I_n - A_n^t\\
-C_n^t
\end{bmatrix}
: \Z^n \rightarrow \Z^{n+1})
@>{\phi_{n,n+1}}>>
\Ker(
\begin{bmatrix}
I_{n+1} - A_{n+1}^t\\
-C_{n+1}^t
\end{bmatrix}
: \Z^{n+1} \rightarrow \Z^{n+2}).
\end{CD}
\end{equation}

is commutative.
\item
Furthermore if $A$ satisfies (DC), then  
the homomorphism \eqref{eq:phin+1n} is isomorphic, 
and hence  
$\iota_{n,n+1}^1: \K_1(\wtT_{A_n})  \rightarrow \K_1(\wtT_{A_{n+1}})$
is isomorphic.
\end{enumerate}
\end{lemma}
\begin{proof}
(i) 
The isomorphism
$ \rho_n^1:\K_1(\wtT_{A_n})  \to 
\Ker(
\begin{bmatrix}
I_n - A_n^t\\
-C_n^t
\end{bmatrix}
: \Z^n \rightarrow \Z^{n+1})
$
given 
in 
\cite[Theorem 3.2]{RaebSzy} 
(cf. \cite[Theorem 3.1]{DrinenTom})
is obtained by  
restricting the isomorphism
between the $\K_0$-groups
$\K_0(\wtT_{A_n})$ and  
$\Coker(
\begin{bmatrix}
I_n - A_n^t\\
-C_n^t
\end{bmatrix}
: \Z^n \rightarrow \Z^{n+1}).
$
As in the proof of \cite[Theorem 3.2]{RaebSzy},
the map
$\iota_{n,n+1}^1:\K_1(\wtT_{A_n}) \to \K_1(\wtT_{A_{n+1}}) 
$
is given by the injection
satisfying
$\iota_{n,n+1}^1([S_iS_i^*]) = [S_iS_i^*], i=1,\dots,n,$
so that the diagram \eqref{eq:CDK1wt} commutes.

(ii) Further assume that  $A$ satisfies (DC), 
so the equalities
$A(i,n+1) = c_i$ for $i=1,\dots, n+1$ hold.
%We then have
%\begin{align*}
%\begin{bmatrix}
%I_{n+1} - A_{n+1}^t\\
%-C_{n+1}^t
%\end{bmatrix}
%=
%\begin{bmatrix}
%1-A(1,1) & -A(2,1) & \dots & -A(n,1)& -A(n+1,1) \\
% -A(1,2) & 1-A(2,2) & \dots& -A(n,2)& -A(n+1,2) \\
%\vdots   & \vdots  &\ddots & \vdots & \vdots \\
% -A(1,n) & -A(2,n) & \dots & 1-A(n,n)& -A(n+1,n) \\
% -A(1,n+1)& -A(2,n+1)& \dots & -A(n,n+1) & 1-A(n+1,n+1)\\
% -A(1,n+1)& -A(2,n+1)& \dots & -A(n,n+1) & -A(n+1,n+1)
% \end{bmatrix}
%\end{align*}
The injectivity of  $\phi_{n,n+1}$ is obvious.
We will show the surjectivity of $\phi_{n,n+1}$.
For any $[y_i]_{i=1}^{n+1} \in 
\Ker(
\begin{bmatrix}
I_{n+1} - A_{n+1}^t\\
-C_{n+1}^t
\end{bmatrix}
: \Z^{n+1} \rightarrow \Z^{n+2}),
$
we have
\begin{align*}
\begin{bmatrix}
1-A(1,1) & -A(2,1) & \dots & -A(n,1)& -A(n+1,1) \\
 -A(1,2) & 1-A(2,2) & \dots& -A(n,2)& -A(n+1,2) \\
\vdots   & \vdots  &\ddots & \vdots & \vdots \\
 -A(1,n) & -A(2,n) & \dots & 1-A(n,n)& -A(n+1,n) \\
 -A(1,n+1)& -A(2,n+1)& \dots & -A(n,n+1) & 1-A(n+1,n+1)\\
 -A(1,n+1)& -A(2,n+1)& \dots & -A(n,n+1) & -A(n+1,n+1)
 \end{bmatrix}
\begin{bmatrix}
y_1\\
\vdots \\
y_n \\
y_{n+1}
\end{bmatrix}
= 0
\end{align*}
so that $y_{n+1} =0$ and hence
\begin{align*}
\begin{bmatrix}
1-A(1,1) & -A(2,1) & \dots & -A(n,1) \\
 -A(1,2) & 1-A(2,2) & \dots& -A(n,2) \\
\vdots   & \vdots  &\ddots & \vdots  \\
 -A(1,n) & -A(2,n) & \dots & 1-A(n,n) \\
 -c_1    & -c_2    & \dots & -c_n 
 \end{bmatrix}
\begin{bmatrix}
y_1\\
\vdots \\
y_n 
\end{bmatrix}
= 0,
\end{align*}
showing that $\phi_{n,n+1}$ is surjective.
\end{proof}
We thus have the following formulas for the $\K$-groups
of the Exel--Laca algebra $\OELA$.
\begin{proposition}\label{prop:KformulaEL}
Assume that a matrix $A$ is (RS) and (LI).
\begin{enumerate}
\renewcommand{\theenumi}{(\roman{enumi})}
\renewcommand{\labelenumi}{\textup{\theenumi}}
\item 
$\K_0(\OELA) \cong 
\varinjlim 
\{ \varphi_{n, n+1}:
\Coker(
\begin{bmatrix}
I_n - A_n^t\\
-C_n^t
\end{bmatrix}
)
\to
\Coker(
\begin{bmatrix}
I_{n+1} - A_{n+1}^t\\
-C_{n+1}^t
\end{bmatrix}
), \, \, 
n \ge K \}
%: \Z^n \rightarrow \Z^{n+1}),
$

$\K_1(\OELA) \cong \varinjlim 
\{ \phi_{n,n+1}:
\Ker(
\begin{bmatrix}
I_n - A_n^t\\
-C_n^t
\end{bmatrix}
)
\to
\Ker(
\begin{bmatrix}
I_{n+1} - A_{n+1}^t\\
-C_{n+1}^t
\end{bmatrix}
),
\, \, n\ge K
\}
%: \Z^n \rightarrow \Z^{n+1}).
$
\item
Furthermore if $A$ satisfies (DC), then  
there exists $K \in \N$ such that
for $n \ge K$, we have   

$\K_0(\OELA) \cong  
\Coker(
\begin{bmatrix}
I_n - A_n^t\\
-C_n^t
\end{bmatrix}
: \Z^n \rightarrow \Z^{n+1}),
$

$\K_1(\OELA) \cong  
\Ker(
\begin{bmatrix}
I_n - A_n^t\\
-C_n^t
\end{bmatrix}
: \Z^n \rightarrow \Z^{n+1}).
$

Hence if the matrix $A$ is (DRS) and (LI), 
then the $\K$-groups of the Exel--Laca algebra $\OELA$ is finitely generated.
\end{enumerate}
\end{proposition}

%\newpage

%%%%%%%%%%%%%%%%%%%%%%%%%%%%%%%%%%%%%%%%%%%%%%%%%%%%%%%%%%%%%%
\section{ $\Exts(\wtT_{A_n})$ }
%%%%%%%%%%%%%%%%%%%%%%%%%%%%%%%%%%%%%%%%%%%%%%%%%%%%%%%%%%%%%%%%%%
We are assuming that an infinite  matrix $A =[A(i,j)]_{i,j\in \N}$ 
is (RS) and  (LI).
Take $K \in \N$ in Definition \ref{def:semifinite} (ii) such that 
$c_i =1$ for some $i \in \{1,\dots,K\}$.
Throughout this section we fix $n\in \N$ with $n \ge K$.
  In this section we will study the extension groups $\Ext_*(\wtT_{A_n})$
 to compute the extension groups $\Ext_*(\OELA)$ for the Exel--Laca algebra $\OELA$. 

Let $H$ be a separable infinite dimensional Hilbert space.
Let us denote by $\B(H)$ the $C^*$-algebra of bounded linear operators on $H$
and $\calQ(H)$ the Calkin algebra $\B(H)/ \calK(H)$. 
The natural quotient map $\B(H) \rightarrow \calQ(H)$ is denoted by $\pi$.
For projections 
$e \in \calQ(H)$ and $E \in \B(H)$ satisfying $\pi(E) = e$,
if an element $u \in \calQ(H)$ satisfies the condition that $eue$ is invertible in $e \calQ(H) e$,
 the Fredholm index $\ind_e u$ is defined. 
 Let $\A$ be a separable unital nuclear $C^*$-algebra.
 An extension $\sigma:\A\rightarrow \calQ(H)$ means a unital $*$-monomorphism.
 The strong (resp. weak) extension group $\Exts(\A)$ (resp. $\Extw(\A)$)  
 is defined to be the set of strong (resp. weak) equivalence classes of extensions 
 $\sigma:\A\rightarrow \calQ(H)$, where two extensions $\sigma, \sigma':\A\rightarrow \calQ(H)$
 are said to be strongly (resp. weakly) equivalent %written $\sigma \sigma_{s} \signa
 if there exists a unitary $U \in \B(H)$ (resp. $u \in \calQ(H)$)
 such that $\sigma = \Ad(\pi(U))\circ \sigma'$ 
 (resp. $\sigma = \Ad(u)\circ \sigma'$) (see \cite{Blackadar}, \cite{BDF}, \cite{HR}).
 It is well-known that $\Exts(\A)$ (resp. $\Extw(\A)$) 
 has a natural semigroup structure
 and becomes a group if $\A$ is nuclear.
%The group $\Exts(\A)$ (resp. $\Extw(\A)$) is regarded as the first strong (resp. weak)  
%extension group $\Exts^1(\A)$ (resp. $\Extw^1(\A)$) 
%as seen in \cite{Blackadar} and \cite{Skandalis}.

As the $C^*$-algebra $\wtT_{A_n}$ is a graph algebra associated to a finite directed graph with a sink,
we may use a general formula for $\Extw$-group studied by Drinen--Tomforde \cite{DrinenTom}
(cf. \cite{TomfordeJOT2003})
so that we have the following.
\begin{lemma}[{cf. \cite[Theorem 3.1]{DrinenTom}}]\label{lem:wExtEL1}
Assume that a matrix $A$ is (RS) and (LI).
Then we have  an isomorphism:
\begin{equation}\label{eq:wExtwtTAn}
\Extw(\wtT_{A_n}) \cong \Coker(
\begin{bmatrix}
I_n - A_n |-C_n
\end{bmatrix}
: \Z^{n+1} \rightarrow \Z^n)
\end{equation}
where
$
 \begin{bmatrix}
I_n - A_n |-C_n
\end{bmatrix}
$
is the $n\times (n+1)$-matrix 
defined by
\begin{equation*}
\begin{bmatrix}
I_n - A_n |-C_n
\end{bmatrix}
= 
{\begin{bmatrix}
1-A(1,1) & -A(1,2) & \dots & -A(1,n)& -c_1 \\
 -A(2,1) & 1-A(2,2) & \dots & -A(2,n)& -c_2 \\
\vdots   & \vdots  &\ddots & \vdots & \vdots \\
 -A(n,1) & -A(n,2) & \dots & 1-A(n,n)& -c_n
\end{bmatrix}}
\end{equation*}
\end{lemma}
Although  the weak extension groups $\Extw(\,\,\, )$
for general graph $C^*$-algebras  including graphs with sinks have been studied in \cite{DrinenTom}
(cf. \cite{RaebSzy}, \cite{TomfordeJOT2003},  etc.),
the  formula for  strong extension groups $\Exts(\,\,\, )$ 
for graph $C^*$-algebras has not been studied.
In this section, we will concentrate on computing the strong extension groups
$\Exts(\wtT_{A_n})$.  
In \cite{MaAnalMath2024}, a formula for  the strong extension group
for a simple Cuntz--Krieger algebra  was presented.
We need a formula of the strong extension group $\Exts(\, \, \, )$ for algebras whose underlying graphs
have sinks.
We will basically follow the discussions given in \cite{MaAnalMath2024}.

Let $\sigma:\wtT_{A_n} \rightarrow \calQ(H)$ be an extension.
Put
$ e_i :=\sigma(S_i S_i^*), i=1,\dots,n$
and
$e_{n+1}:= \sigma(P_{n+1}).$ 
Let $\tau:\wtT_{A_n} \rightarrow \calQ(H)$ 
be a trivial extension such that 
\begin{equation*}
 \tau(S_j S_j^*) =\sigma(S_i S_i^*), \quad j=1,\dots,n
 \quad \text{ and } \quad \tau(P_{n+1})= \sigma(P_{n+1}).
\end{equation*}
Note that the above $\tau$ exists because $\Exts (\mathbb{C}^{n+1})=0$.
Since 
$\sigma(S_i)\tau(S_i^*)$ is a unitary in $\calQ(H)$
commuting with $e_i$, the Fredholm index 
\begin{equation}\label{eq:indexdi} 
d_i(\sigma,\tau) = \ind_{e_i}\sigma(S_i)\tau(S_i^*), \qquad i=1,\dots,n
\end{equation}
is defined.
Let $\tau_m:\wtT_{A_n} \rightarrow \calQ(H), m=1,2$ 
be two trivial extensions such that 
\begin{equation*}
 \tau_m(S_j S_j^*) =\sigma(S_i S_i^*), \quad j=1,\dots,n
 \quad \text{ and } \quad
 \tau_m(P_{n+1})= \sigma(P_{n+1}), \quad \, \, m=1,2.
\end{equation*} 
By \cite{Voiculescu}, one may find a unitary 
$U \in \B(H)$ such that 
$\tau_1(x ) =\pi(U) \tau_2(x) \pi(U^*),\;\; x \in \wtT_{A_n}$.
The condition $\tau_1={\rm Ad}\;\pi(U)\circ \tau_2$ implies that $\pi(U)$ commutes with $e_i$, and one can define
$k_i = \ind_{e_i} \pi(U), i=1,\dots, n+1$.
By a similar method to the proof of \cite[Lemma 2.2]{MaAnalMath2024}, 
we have the following lemma.
%The following lemma is proved by a completely simimalr manner to 
%\cite[Lemma 2.2]{MaAnalMath2024}, so we omit its proof. 
\begin{lemma}[{cf. \cite[Lemma 2.2]{MaAnalMath2024}}] \label{lem:Exts1}
%Let $\sigma:\wtT_{A_n} \rightarrow \calQ(H)$ be an extension.
%For two trivial extensions $\tau_m:\wtT_{A_n} \rightarrow \calQ(H), m=1,2$ 
%satisfying  
%$\tau_m(P_{n+1})= \sigma(P_{n+1})$
%and
%$
% \tau_m(S_j S_j^*) =\sigma(S_i S_i^*), \, \, j=1,\dots,n,
%$
%there exists $[k_i]_{i=1}^{n+1} \in \Z^{n+1}$ such that
Keep the above notation. We have
$$
d_i(\sigma,\tau_1) -d_i(\sigma, \tau_2)
= \sum_{j=1}^n A(i,j)k_j - k_i + c_i k_{n+1}, \, \, \, i=1,\dots,n
\quad
\text{ and }
\quad
\sum_{j=1}^{n+1} k_j =0.
$$
\end{lemma}
%\begin{proof}
%(i) By the Voiculescu's lemma, one may find a unitary $U \in \B(H)$ such that 
%$\tau_2(x) = \pi(U) \tau_1(x) \pi(U)^*, x \in \wtT_{A_n}$.
%Put $u := \pi(U)$. Since $e_{i}, i=1,\dots, n+1$ commute with $u$, we have
%$(e_i u e_i)^*(e_i u e_i) = (e_i u e_i)(e_i u e_i)^* = e_i$ for $i=1,\dots,n+1$. 
%Put $k_i: = \ind_{e_i}u, i=1,\dots, n+1.$
%For $i=1,\dots,n$, we then have
%\begin{align*}
%d_i(\sigma,\tau_2)
%=& \ind_{e_i}\sigma(S_i) \tau_2(S_i^*)\\ 
%=& \ind_{e_i}\sigma(S_i) \tau_1(S_i^*S_i) u \tau_1(S_i^*S_i) \tau_1(S_i^*) u^* \\ 
%=& \ind_{e_i}\sigma(S_i) \tau_1(S_i^*) \tau_1(S_i) u \tau_1(S_i^*) e_i u^* \\ 
%=& \ind_{e_i}\sigma(S_i) \tau_1(S_i^*) 
%+ \ind_{e_i}\tau_1(S_i) u \tau_1(S_i^*)  +\ind_{e_i} u^* \\ 
%= & d_i(\sigma, \tau_1) + \ind_{e_i}\tau_1(S_i) u \tau_1(S_i^*) -k_i.
%\end{align*}
%Since $\sigma(S_i^* S_i ) = \sum_{j=1}^n A(i,j) \sigma(S_j S_j^*) + c_i \sigma(P_{n+1})
%=  \sum_{j=1}^n A(i,j) e_j + c_i e_{n+1},$
%we have
%\begin{align*}
%\ind_{e_i}\tau_1(S_i) u \tau_1(S_i^*)
%=&  \ind_{\tau_1(S_i^* S_i)} u
%=  \ind_{\sigma(S_i^* S_i)} u \\
%=& \sum_{j=1}^n A(i,j) \ind_{e_j} u + c_i\ind_{e_{n+1}} u
%= \sum_{j=1}^n A(i,j) k_j + c_i k_{n+1}
%\end{align*}
%so that 
%\begin{align*}
%d_i(\sigma,\tau_2)
%= & d_i(\sigma, \tau_1) + \sum_{j=1}^n A(i,j) k_j + c_i k_{n+1} -k_i.
%\end{align*}
%(ii)
%Since $\sum_{j=1}^n S_j S_j^* + P_{n+1} =1$
%and $u = \pi(U)$ with $U$ being a unitary,
%we have
%$$
%\sum_{j=1}^{n+1} k_j = \sum_{j=1}^{n+1}\ind_{e_j} u
% = \ind_{\sum_{j=1}^{n+1} e_j} u =0.
% $$
%\end{proof}
%%%%%%%%%%%%%%%%%%%%%%%%%%%%%%%%%%%%
By Lemma \ref{lem:Exts1}, we have
\begin{equation}\label{eq:dsigmak}
\begin{bmatrix}
d_1(\sigma, \tau_1)\\
%d_2(\sigma, \tau_1)\\
\vdots             \\
d_n(\sigma, \tau_1)\\
\end{bmatrix}
-
\begin{bmatrix}
d_1(\sigma, \tau_2)\\
%d_2(\sigma, \tau_2)\\
\vdots             \\
d_n(\sigma, \tau_2)\\
\end{bmatrix}
=-
\begin{bmatrix}
 &          & -c_1 \\
 %&  & -c_2 \\
 & I_n -A_n & \vdots \\
 &          & -c_n \\
\end{bmatrix}
\begin{bmatrix}
k_1\\
%k_2\\
\vdots \\
k_n \\
k_{n+1}
\end{bmatrix}
\quad
\text{ and }
\quad
\sum_{j=1}^{n+1} k_j =0.
\end{equation}
Let us denote by $R_{n+1}$ the $(n+1)\times (n+1)$ matrix
whose bottom row is the vector $[1,\dots,1]$ and the other rows are zero vectors,
so that  
\begin{equation*}
\{ [k_i]_{i=1}^{n+1} \in \Z^{n+1} \mid \sum_{i=1}^{n+1} k_i =0 \}
=\{ (I_{n+1} - R_{n+1})[l_i]_{i=1}^{n+1} \in \Z^{n+1} \mid [l_i]_{i=1}^{n+1} \in \Z^{n+1}\}.
\end{equation*}
Then the formula \eqref{eq:dsigmak} says that 
for an extension
$\sigma:\wtT_{A_n}\rightarrow \calQ(H)$, 
one may define  
\begin{equation}\label{eq:dsigma}
d_\sigma:= [d_i(\sigma,\tau)]_{i=1}^n  \in \Z^{n}/{ [I_n - A_n \mid -C_n][I_{n+1} - R_{n+1}] \Z^{n+1}}.
\end{equation}
It is easy to verify that 
the value $d_\sigma$ does not depend on the choice of the strong equivalence class of 
an extension
$\sigma: \wtT_{A_n}\rightarrow \calQ(H)$ so that 
$d_\sigma$ 
may be defined on the class $[\sigma]_s \in \Exts(\wtT_{A_n})$
(cf. \cite[Lemma 2.3]{MaAnalMath2024}).
Define the $n \times n$ matrix 
$\widetilde{A}_{C_n}$ by setting
\begin{equation}\label{eq:ACn}
\widetilde{A}_{C_n}:
=
\begin{bmatrix}
A(1,1) - c_1 & A(1,2) -c_1 & \dots & A(1,n) - c_1  \\
A(2,1) - c_2 & A(2,2) -c_2 & \dots & A(2,n) - c_2  \\
\vdots       & \vdots      & \ddots& \vdots  \\
A(n,1) - c_n & A(n,2) -c_n & \dots & A(n,n) - c_n    
\end{bmatrix}.
\end{equation}
We note that the entries of $\widetilde{A}_{C_n}$ possibly negative integers.
We write the quotient group
$\Z^n/ (I_n -\widetilde{A}_{C_n})\Z^n$ 
as
$\Cokern.$ 
Since
\begin{align*}
  & [I_n - A_n \mid -C_n][I_{n+1} - R_{n+1}] \\ 
%= &
%{
%\begin{bmatrix}
%         & -c_1 \\
%I_n -A_n & \vdots \\
%         & -c_n
%\end{bmatrix}
%I_{n+1}
%-
%\begin{bmatrix}
%         & -c_1 \\
%I_n -A_n & \vdots \\
%         & -c_n
%\end{bmatrix}
%\begin{bmatrix}
%0 & \dots & 0\\
%\vdots &  & \vdots\\
%0 & \dots & 0\\
%1 & \dots & 1
%\end{bmatrix} 
%} \\
%=&
%{
%\begin{bmatrix}
%         & -c_1 \\
%I_n -A_n & \vdots \\
%         & -c_n
%\end{bmatrix}
%I_{n+1}
%-
%\begin{bmatrix}
%-c_1 & \dots & -c_1\\
%-c_2 & \dots & -c_2\\
%\vdots &  & \vdots \\
%-c_n & \dots & -c_n
%\end{bmatrix} 
%} \\
=&
{
\begin{bmatrix}
1 - A(1,1) + c_1 & -A(1,2) +c_1 & \dots & -A(1,n) + c_1 & 0 \\
  - A(2,1) + c_2 &1-A(2,2) +c_2 & \dots & -A(2,n) + c_2 & 0 \\
\vdots           & \vdots       &       & \vdots        & \vdots \\
  - A(n,1) + c_n & -A(n,2) +c_n & \dots &1-A(n,n) + c_n & 0   
\end{bmatrix},
}
\end{align*}
we have
\begin{equation*}
\Z^{n}/{ [I_n - A_n \mid -C_n][I_{n+1} - R_{n+1}] \Z^{n+1}}
= \Cokern,
\end{equation*}
so that $d_\sigma \in \Cokern$
for an extension $\sigma: \wtT_{A_n} \rightarrow \calQ(H)$
because of \eqref{eq:dsigma}.

By a completely similar manner to the proof of \cite[Proposition 2.4]{MaAnalMath2024},
we may prove the following proposition. %We omit its proof.
\begin{proposition}[{cf. \cite[Proposition 2.4]{MaAnalMath2024}}] \label{prop:Exts3}
The map defined by
\begin{equation*}
d_{A_n} : [\sigma]_s \in \Exts(\wtT_{A_n}) 
\rightarrow d_\sigma \in \Cokern
\end{equation*}
is an isomorphism of groups, and the map induces an isomorphism
$\Extw(\wtT_{A_n}) \cong \Coker([I_n - A_n \mid -C_n]: \Z^{n+1}\rightarrow \Z^n)$
in \eqref{eq:wExtwtTAn}.
\end{proposition}
%\newpage

%%%%%%%%%%%%%%%%%%%%%%%%%%%%%%%%
\section{ $\Ext_*(\OELA)$}
%%%%%%%%%%%%%%%%%%%%%%%%%%%%%%%%%%
We are assuming that the infinite matrix $A = [A(i,j)]_{i,j\in \N}$
is (RS) and (LI).
Keep the notation as in the preceding sections.
For an extension $\sigma:\OELA \rightarrow Q(H)$,
we write its restriction 
$\sigma|_{\wtT_{A_n}}: \wtT_{A_n} \rightarrow Q(H)$
to the subalgebra $\wtT_{A_n}$ as $\sigma_n$.
%As the class $[\sigma]_s \in \Exts(\OELA)$
%gives rise to a well-defined class
%$[\sigma_n]_s \in \Exts(\wtT_{A_n})$,
We then have a homomorphism
$
[\sigma]_s \in \Exts(\OELA) 
 \rightarrow [\sigma_n]_s \in \Exts(\wtT_{A_n}).
 $
 Similarly the inclusion 
 $\wtT_{A_n} \hookrightarrow \wtT_{A_{n+1}}$
 defines a natural homomorphism 
 $
\Exts(\wtT_{A_{n+1}}) \rightarrow \Exts(\wtT_{A_n})
$  written as $\iota_{n+1,n}^s$ 
by restriction of extensions of $\wtT_{A_{n+1}}$
to $\wtT_{A_n}$.
We note the following lemma.
\begin{lemma}
%Assume that the infinite matrix $A = [A(i,j)]_{i,j\in \N}$
%is (RSF), (WUT) and (LI).
The restriction
$[d_i]_{i=1}^{n+1}\in \Z^{n+1}\rightarrow [d_i]_{i=1}^{n}\in \Z^{n}$
induces a homomorphism
$
\Cokernplus1\rightarrow
\Cokern
$
written as $I_{n+1,n}^s$
\end{lemma}
\begin{proof}
For $[d_i]_{i=1}^{n+1}\in (I_{n+1} -\widetilde{A}_{C_{n+1}})\Z^{n+1}$, take 
$[k_i]_{i=1}^{n+1} \in \Z^{n+1}$
 such that 
\begin{equation}\label{eq:diki}
d_i = k_i -\sum_{j=1}^{n+1} A(i,j) k_j + c_i \sum_{j=1}^{n+1} k_j, \quad i=1,\dots, n+1.
\end{equation} 
By (RS), we know 
$ A(i,n+1) = c_i$
for
$i=1,\dots, n,$
so that 
\begin{equation}\label{eq:diki3}
d_i = k_i -\sum_{j=1}^{n} A(i,j) k_j + c_i \sum_{j=1}^n k_j, \quad i=1,\dots, n,
\end{equation} 
showing that 
$[d_i]_{i=1}^{n}\in (I_{n} -\widetilde{A}_{C_{n}})\Z^n.$
Hence 
the restriction
$[d_i]_{i=1}^{n+1}\in \Z^{n+1}\rightarrow [d_i]_{i=1}^{n}\in \Z^{n}$
induces a homomorphism
$
\Cokernplus1 \rightarrow
\Cokern.
$
\end{proof}
We then have the following.
\begin{lemma}\label{lem:CDstrong}
Assume that a matrix $A$ is (RS) and (LI).
\begin{enumerate}
\renewcommand{\theenumi}{(\roman{enumi})}
\renewcommand{\labelenumi}{\textup{\theenumi}}
\item 
The diagram
\begin{equation}\label{eq:CDstrong}
\begin{CD}
\Exts(\wtT_{A_{n+1}}) @>{\iota_{n+1,n}^s}>> \Exts(\wtT_{A_{n}}) \\
@V{d_{A_{n+1}}}VV @VV{d_{A_{n}}}V \\
\Cokernplus1 @>{I_{n+1,n}^s}>>
\Cokern
\end{CD}
\end{equation}
is commutative.
\item
Furthermore if $A$ satisfies (DC), then  
the homomorphism 
$I_{n+1,n}^s:
\Cokernplus1 \rightarrow
\Cokern
$
is isomorphic, 
and hence  
$\iota_{n+1,n}^s: \Exts(\wtT_{A_{n+1}}) \rightarrow \Exts(\wtT_{A_{n}})$
is isomorphic.
\end{enumerate}
\end{lemma}
\begin{proof}
(i)
As the map
$d_{A_{n+1}}: \Exts(\wtT_{A_{n+1}}) \rightarrow \Cokernplus1$
is given by 
$d_{A_{n+1}}([\sigma]_s) = [d_i(\sigma, \tau)]_{i=1}^{n+1}$
for an extension
$\sigma:\wtT_{A_{n+1}}\rightarrow \calQ(H)$,
the homomorphism
$\iota_{n+1,n}^s:  \Exts(\wtT_{A_{n+1}}) \rightarrow \Exts(\wtT_{A_n})$
is given by the restriction 
$
[d_i(\sigma, \tau)]_{i=1}^{n+1}\rightarrow [d_i(\sigma, \tau)]_{i=1}^{n}.$
This shows that the diagram \eqref{eq:CDstrong} is commutative.

(ii)
%We are assuming that $A(i,n+1) = c_i$ for $i=1,\dots,n$,
%and furthermore (UT) admits the condition 
%$A(n+1,n+1) = c_{n+1}$.
For $[d_i]_{i=1}^{n+1} \in \Z^{n+1}$,
suppose that 
$I_{n+1,n}^s([d_i]_{i=1}^{n+1}) = 0$
in $\Cokern,$ 
which means that 
 one may find $[k'_i]_{i=1}^{n} \in \Z^{n}$
 such that 
\begin{equation}\label{eq:di4.14}
d_i = k'_i - \sum_{j=1}^n A(i,j) k'_j + c_i \sum_{j=1}^n  k^\prime_j, \qquad i=1,\dots,n. 
 \end{equation}
By putting 
$k_i: =  k'_i$
for $ i=1,\dots, n$ and 
\begin{equation}
k_{n+1} :=  d_{n+1} +  \sum_{j=1}^n A(n+1,j) k'_j - c_{n+1} \sum_{j=1}^n k_j' \label{eq:kn+1}
\end{equation}
and using the equalities
$A(i,n+1) = c_i$ for $i=1,\dots,n$ from (RS),
 we have for $i=1,\dots, n$
$$
d_i 
=   k_i - \sum_{j=1}^{n+1} A(i,j) k_j  +c_i \sum_{j=1}^{n+1} A(i,j) k_j.
$$  
For $i=n+1$, the condition $A(n+1, n+1) = c_{n+1}$ from (DC) with \eqref{eq:kn+1} 
implies
 \begin{align*}
d_{n+1} 
%= & k_{n+1} - \sum_{j=1}^{n} A(n+1,j) k'_j + c_{n+1} \sum_{j=1}^n k'_j \\
=  k_{n+1} - \sum_{j=1}^{n+1} A(n+1,j) k_j + c_{n+1} \sum_{j=1}^{n+1} k_j, 
 \end{align*}
showing 
that $I_{n,n+1}^s: 
\Cokernplus1 \rightarrow
\Cokern
$
is injective.
As the surjectivity of the map
$I_{n+1,n}^s:
\Cokernplus1 \rightarrow
\Cokern
$
is obvious, the map $I_{n+1, n}$ is isomorphic.
\end{proof}
We similarly have homomorphisms
$$
\Extw(\wtT_{A_{n+1}}) \rightarrow 
\Extw(\wtT_{A_n}), \quad
\Coker([I_{n+1} - A_{n+1} | -C_{n+1}]) \to
\Coker([I_{n} - A_{n} | -C_{n}])
$$
which are  written as $\iota_{n+1,n}^w, \, I_{n n+1}^w$, respectively,
 to obtain the following lemma.
%By Proposition \ref{prop:Exts3},
%we similarly have the following lemma.
\begin{lemma}\label{lem:CDweak}
Assume that a matrix $A$ is (RS) and (LI).
\begin{enumerate}
\renewcommand{\theenumi}{(\roman{enumi})}
\renewcommand{\labelenumi}{\textup{\theenumi}}
\item 
The diagram
\begin{equation}\label{eq:CDweak}
\begin{CD}
\Extw(\wtT_{A_{n+1}}) @>{\iota_{n+1,n}^w}>> \Extw(\wtT_{A_{n}}) \\
@V{d_{A_{n+1}}}VV @VV{d_{A_{n}}}V \\
\Z^{n+1}/[I_{n+1} - A_{n+1} | -C_{n+1}]\Z^{n+2} 
@>{I_{n+1,n}^w}>> 
\Z^{n}/[I_n -A_n | -C_n] \Z^{n+1}
\end{CD}
\end{equation}
is commutative.
\item
Furthermore if $A$ satisfies (DC), then  
the homomorphism of the bottom arrow in the diagram
\eqref{eq:CDstrong}
is isomorphic, 
and hence  
$\iota_{n+1,n}^w: \Extw(\wtT_{A_{n+1}}) \rightarrow \Extw(\wtT_{A_{n}})$
is isomorphic.
\end{enumerate}
\end{lemma}
 It is well-known that 
 for a separable unital nuclear $C^*$-algebra $\A$, one may
  identify the $\Exts(\A)$-group (resp. $\Extw(\A)$-group)
 with the first strong (resp. weak) extension group
 $\Exts^1(\A)$ (resp. $\Extw^1(\A)$) (see \cite{Blackadar}, \cite{Skandalis}).
 % Let us recall the cyclic six term exact sequence for 
 %$\Ext$-groups for a separable unital nuclear $C^*$-algebras $\A$.
 The strong extension groups $\Exts^i(\A)$ and the weak extension groups
 $\Extw^i(\A)$  are recognized as the Kasparov's $\sqK$-groups
 such that
 \begin{equation*}
 \Exts^i(\A) = \sqK^{i+1}(C_\A, \mathbb{C}), \qquad
 \Extw^i(\A) = \sqK^{i}(\A, \mathbb{C}), \qquad
 i=0,1, 
 \end{equation*}
 where
 $C_\A$ is the mapping cone algebra
 $\{ x \in C(0,1]\otimes \A \mid x(1) \in \mathbb{C} 1_\A\}$ (see \cite{Blackadar}, \cite{Skandalis}).
 By applying the $\sqK( \, \, \, , \mathbb{C})$-functor to the short exact sequence
 $
 0 \rightarrow C_0(0,1)\otimes \A \rightarrow C_\A \rightarrow \A \rightarrow 0,
 $ 
we have a cyclic six term exact sequence 
 \begin{equation}\label{eq:6termExtforA}
\begin{CD}
\Exts^0(\A) @>>> \Extw^0(\A) @>>>  \Z \\
@AAA @. @V{\iota_{\A}}VV \\
0 @<<<  \Extw^1(\A) @<<< \Exts^1(\A)
\end{CD}.
\end{equation}
By using the above observation, we have the following.
\begin{proposition}\label{prop:mainExtsExtw1}
Assume that the infinite matrix $A = [A(i,j)]_{i,j\in \N}$
is (DRS) and (LI).
Let $A_n=[A(i,j)]_{i,j=1}^n$ be the $n \times n$ matrix defined by $A$.
Define the $n \times n$ matrix 
$
\widetilde{A}_{C_n}
$ by
\eqref{eq:ACn}.
Then we have for $n \ge K$
\begin{enumerate}
\renewcommand{\theenumi}{(\roman{enumi})}
\renewcommand{\labelenumi}{\textup{\theenumi}}
%\item $\Exts(\OELA) 
%\cong \varprojlim \Exts(\wtT_{A_n}) 
%\cong  \varprojlim \Z^n / ( I_n - \widetilde{A}_{C_n}) \Z^n,$ %
% $\Extw(\OELA) 
%\cong  \varprojlim \Extw(\wtT_{A_n}) 
%\cong  \varprojlim \Z^n / [I_n -A_n | -C_n]\Z^{n+1}.$ %%
%
%\item Furthermore if $A$ satisfies (DC) and hence (UT),
%then there exists $N \in \N$ such that for $n \ge N$
\item $\Extw(\OELA) 
\cong \Extw(\wtT_{A_n}) 
\cong  \Z^n / [I_n -A_n | -C_n]\Z^{n+1}.$ 
\item $\Exts(\OELA) 
\cong  \Exts(\wtT_{A_n}) 
\cong  \Z^n / ( I_n - \widetilde{A}_{C_n}) \Z^n.$ 
\end{enumerate}
\end{proposition}
\begin{proof}
 Assume that the infinite matrix 
 $A = [A(i,j)]_{i,j\i \N}$
is  (DRS) and (LI).
Note that the $\K$-groups $\K_*(\OELA)$ of $\OELA$ 
are finitely generated by Proposition \ref{prop:KformulaEL}.
 We identify $\Ext_{*}(\OELA)$ with $\Ext_*^1(\OELA)$.
% Hence by the UCT, the weak extension group
% $\Extw(\OELA)$ is finitely generated.
% Since there exists a  sequence
% $\Z\rightarrow \Exts(\OELA) \rightarrow \Extw(\OELA)$
% which is exact at the middle, 
% $\Exts(\OELA)$ is also finitely generated.

Since $\Extw^{i}(\OELA)\cong \sqK^i(\OELA), \mathbb{K}), \;(i=0, 1)$ are countable groups by the UCT (cf.\cite{Blackadar}), \cite[Proposition 3.5]{Schochet}, \cite[Theorem 21.5.2]{Blackadar} and Lemma \ref{lem:CDweak} show that
 \[\Extw^i(\OELA)=\varprojlim \Extw^i(\wtT_{A_n})\cong \Extw^i(\wtT_{A_n})
 \quad\text{ for } n \ge K,\]
showing the formula (i).
 Using the exact sequences \eqref{eq:6termExtforA} for  $\OELA$ and $\wtT_{A_n}$,
 one has the following commutative diagram
 \[\xymatrix{
\Extw^0(\OELA)\ar[r]\ar[d]^{\cong}&\mathbb{Z}\ar@{=}[d]\ar[r]&\Ext_s(\OELA)\ar[d]\ar[r]&\Extw(\OELA)\ar[d]^{\cong}\ar[r]&0\\
\Extw^0(\wtT_{A_n})\ar[r]&\mathbb{Z}\ar[r]&\Exts(\wtT_{A_n})\ar[r]&\Extw(\wtT_{A_n})\ar[r]&0
 }\]
 with exact horizontal sequences.
Thus, the middle vertical arrow
\[\Exts(\OELA)\ni [\sigma]_s\mapsto [\sigma |_{\wtT_{A_n}}]_s\in \Exts (\wtT_{A_n})\]
is an isomorphism, showing the formula (ii).
 %Hence 
 %the Milner $\varprojlim^1$-groups
 %$\varprojlim^1\Exts(\wtT_{A_n})$
 %and $\varprojlim^1\Extw(\wtT_{A_n})$
 %both vanish (cf. \cite[Proposition 3.5]{Schochet}), so that 
% we have
 %$\Exts(\OELA)=\varprojlim \Exts(\wtT_{A_n})$
 %and
% $\Extw(\OELA)=\varprojlim \Extw(\wtT_{A_n})$
 %(cf. \cite[Proposition 3.5]{Schochet}).
%Now the matrix $A$ is (DRS), so it satisfies (DC).
%By Lemma \ref{lem:CDstrong} and Lemma \ref{lem:CDweak}, 
%we have the desired formulas (i) and (ii). 
\end{proof}

%%%%%%%%%%%%%%%%%%%%%%%%%%%%%%%%%%%%%%%%%%%%%%%%%%%%%%%%%%%%%%%%%%%%%%%%%%%%%%%%%%%%%%%%%%%%%%%%%%
\section{ The cyclic six-term exact sequence for $\Ext_*(\OELA)$} \label{sect:cyclicsixtermforExt}
%%%%%%%%%%%%%%%%%%%%%%%%%%%%%%%%%%%%%%%%%%%%%%%%%%%%%%%%%%%%%%%%%%%%%%%%%%%%%%%%%%%%%%%%%%%%%%%%
 In this section, we will compute the cyclic six term exact sequence 
 \eqref{eq:6termExtforA} for the Exel--Laca algebra $\OELA$.
The right down arrow in \eqref{eq:6termExtforA}
is denoted by $\iota_A$ for $\A = \OELA$.
 The computation plays a key role to detect the reciprocal dual
 $\widehat{\OELA}$ of $\OELA$. 
 
  %%%%%%%%%%%%%%%%%%%%%%%%%%%%%%%%%%%%%%%%%%%%%%%%%%%%
  
 Assume that the matrix $A$ is (RS) and (LI).
  We fix $n \in \N$ with $n \ge K$.
 For $m \in \Z$, take a unitary $u_m \in Q(H)$ of Fredholm index $m$.
 Take a trivial extension $\tau:\wtT_{A_n} \rightarrow Q(H)$, 
 and consider an extension
 $\sigma_m : = \Ad(u_m) \circ \tau : \wtT_{A_n} \rightarrow Q(H)$.
 Then the map $\iota_{A_n} : \Z \rightarrow \Exts^1(\wtT_{A_n})$
 defined by 
 $\iota_{A_n}(m) := [\sigma_m]_s \in \Exts^1(\wtT_{A_n})$
 gives rise to  a homomorphism
 such that
 $\Exts^1(\wtT_{A_n})/\Z \iota_{A_n}(1) \cong \Extw^1(\wtT_{A_n})$.
The following lemma may be proved by a completely similar fashion to the proof
of \cite[Proposition 3.1]{MaAnalMath2024}.
 \begin{lemma}[{\cite[Proposition 3.1]{MaAnalMath2024}}]\label{lem:iotahat}
 Define
 $\hat{\iota}_{A_n}: \Z \rightarrow \Cokern$
 by setting
 $\hat{\iota}_{A_n}(m) := [I_n -A_n |-C_n ] [k_i]_{i=1}^{n+1}$ 
 with 
 $ m= \sum_{i=1}^{n+1}k_i.$ 
 Then we have
\begin{enumerate}
\renewcommand{\theenumi}{(\roman{enumi})}
\renewcommand{\labelenumi}{\textup{\theenumi}}
\item $\hat{\iota}_{A_n}(m)$ 
does not depend on the choice of $[k_i]_{i=1}^{n+1}$ as long as 
 $ m= \sum_{i=1}^{n+1}k_i.$ 
\item The diagram
\begin{equation}
\begin{CD}
\Z @>{\iota_{A_n}}>> \Exts^1(\wtT_{A_n}) \\ 
\| @.  @VV{d_{A_n}}V \\
\Z @>{\hat{\iota}_{A_n}}>>  \Cokern  
\end{CD}
\end{equation}
is commutative, that is, $d_{A_n}(\iota_{A_n}(m)) = \hat{\iota}_{A_n}(m)$
for $m \in \Z$.
%where $\iota_{A_n}(m) = [\sigma_m]_s$.
\end{enumerate} 
 \end{lemma}
We know the following lemma.
\begin{lemma}\label{lem:iotaAn2}
Assume that the matrix $A$ is (RS) and (LI).
\begin{enumerate}
\renewcommand{\theenumi}{(\roman{enumi})}
\renewcommand{\labelenumi}{\textup{\theenumi}}
\item
The diagrams 
\begin{equation}\label{eq:CDExtsCoker1}
\begin{CD}
\Z @>{\iota_{A_{n+1}}}>> \Exts^1(\wtT_{A_{n+1}}) 
 @>{d_{A_{n+1}}}>> \Cokernplus1 \\
\| @.  @V{\iota_{n+1,n}^s}VV @V{I_{n+1,n}^s}VV  \\ 
\Z @>{\iota_{A_n}}>> \Exts^1(\wtT_{A_{n}}) 
 @>{d_{A_n}}>> \Cokern \\
\end{CD}
\end{equation}
are commutative.
\item
\begin{equation}\label{eq:Exts1iotahatA1}
(\Exts^1(\wtT_{A_n}), \iota_{A_n}(1))
=
(\Cokern, [-C_n]).
\end{equation}
\item 
Furthermore if $A$ satisfies (DC), then
\begin{equation}\label{eq:EXTsiotaA1}
(\Exts^1(\OELA), \iota_{A}(1))
=
(\Cokern, [-C_n]).
\end{equation}
\end{enumerate}
\end{lemma}
\begin{proof}
(i) The commutativity of the left square in \eqref{eq:CDExtsCoker1} is direct.
The commutativity of the right square comes from Lemma \ref{lem:CDstrong}.

(ii)
Put $k_i =0$ for $i=1,\dots, n$ and $k_{n+1} =1$, so that 
$m:= \sum_{i=1}^{n+1} k_i =1$.
We then have
\begin{equation*}
\hat{\iota}_{A,n}(1) = [I_n -A_n | -C_n][k_i]_{i=1}^{n+1} = [-C_n]
\end{equation*}
so that we get the desired assertion.

(iii)
Since the map $I_{n,n+1}^s: \Z^{n+1}\rightarrow \Z^n$ 
is defined by removing the bottom entry of the vectors in $\Z^{n+1}$,
we have 
$I_{n,n+1}^s([-C_{n+1}]) = [-C_n]$
and  the desired equality \eqref{eq:EXTsiotaA1}.
\end{proof}
It is direct to see the following lemma by linear algebra.
\begin{lemma}\label{lem:KerwAnn+1}
We have the cyclic six term exact sequence:
\begin{equation}\label{eq:6termAnn+1}
\begin{CD}
\Ker(I_n - \widetilde{A}_{C_n}) @>{j_{A_n}}>> \Ker([I_n - A_n | -C_n]) @>{s_{A_n}}>>  \Z \\
@AAA @. @V{\hat{\iota}_{A_n}}VV \\
0 @<<< \Z^n/[I_n -A_n|-C_n]\Z^{n+1} @<<< \Cokern
\end{CD}
\end{equation}
where
%\hat{\iota}_{A_n}:  \Z \rightarrow \Cokern, \qquad
$s_{A_n}: \Ker([I_n - A_n | -C_n]) \rightarrow \Z 
$ and  
%\iota_1:  \Z \rightarrow \Ker(I_n - \widetilde{A}_{C_n}), \qquad
$
j_{A_n} :  \Ker(I_n - \widetilde{A}_{C_n}) \rightarrow \Ker([I_n - A_n | -C_n]) 
$
are defined by
\begin{gather*}
%\hat{\iota}_A(m) 
%:=  [I_n - A_n | -C_n] 
%\begin{bmatrix}
%k_1\\
%\vdots \\
%k_{n+1}
%\end{bmatrix}
%\text{ for } m = \sum_{i=1}^{n+1} k_i, \qquad
s_{A_n}([l_i]_{i=1}^{n+1}) 
:= \sum_{i=1}^{n+1} l_i
\quad
\text{and}
\quad
%\iota_1(n)
%:=  
%\begin{bmatrix}
%0\\
%\vdots \\
%0 \\
%n
%\end{bmatrix} \in \Z^{n+1},\qquad 
j_{A_n}([l_i]_{i=1}^{n}) 
:=   
\begin{bmatrix}
l_1\\
\vdots \\
l_n \\
-\sum_{i=1}^n l_1
\end{bmatrix} \in \Z^{n+1}, 
\end{gather*}
respectively, and the bottom map
$\Z^n/[I_n -A_n|-C_n]\Z^{n+1} \leftarrow \Cokern
$
is the homomorphism induced by the identity on $\Z^n$.
\end{lemma}

\begin{lemma}
$\Extw^0(\wtT_{A_n}) \cong \Ker([I_n -A_n | -C_n] :\Z^{n+1}\rightarrow \Z).$
\end{lemma}
\begin{proof}
As $\Extw^0(\wtT_{A_n}) = \sqK(\wtT_{A_n} ,\mathbb{C})$,
the UCT tells us a short exact sequence:
\begin{equation}
0
\rightarrow
\Ext_\Z^1(\K_1(\wtT_{A_n}), \Z)
\rightarrow 
\Extw^0(\wtT_{A_n})
\rightarrow 
\Hom(\K_0(\wtT_{A_n}), \Z)
\rightarrow 0.
\end{equation}
By Lemma \ref{lem:EL1}, the group $\K_1(\wtT_{A_n})$ is torsion free
and hence
we have
\begin{equation*}
\Extw^0(\wtT_{A_n}) 
\cong \Hom(\K_0(\wtT_{A_n}), \Z)
\cong \Hom(\Coker(
\begin{bmatrix}
I_n -A_n^t \\
-C_n
\end{bmatrix}: \Z^n \rightarrow \Z^{n+1}), \mathbb{Z}).
\end{equation*}
By elementary linear algebra,
the last group above is isomorphic to
%\begin{equation*}
%\Hom(
%\Coker(
%\begin{bmatrix}
%I_n -A_n^t \\
%-C_n
%\end{bmatrix}: \Z^n \rightarrow \Z^{n+1}), \Z)
%\cong
$
\Ker([I_n -A_n | -C_n] :\Z^{n+1}\rightarrow \Z).
$
\end{proof}
Hence we have the following.
\begin{proposition}\label{prop:6termforExtwtTAn}
The cyclic six term exact sequence \eqref{eq:6termExtforA}
for the $\Ext$-groups of the $C^*$-algebra $\wtT_{An}$ 
\begin{equation}\label{eq:6termforExtwtTAn}
\begin{CD}
\Exts^0(\wtT_{A_n}) @>{j_{A_n}}>> \Extw^0(\wtT_{A_n}) @>{s_{A_n}}>>  \Z \\
@AAA @. @V{\hat{\iota}_{A_n}}VV \\
0 @<<<  \Extw^1(\wtT_{A_n}) @<<< \Exts^1(\wtT_{A_n})
\end{CD}
\end{equation}
is computed to be the cyclic six term exact sequence
\eqref{eq:6termAnn+1}.
\end{proposition}
\begin{proof}
Since the three groups $\Extw^0(\wtT_{A_n}), \Extw^1(\wtT_{A_n}), \Exts^1(\wtT_{A_n})$ 
determine the other group $\Exts^0(\wtT_{A_n})$ through the 
cyclic six term exact sequence \eqref{eq:6termforExtwtTAn},
the two cyclic six term exact sequences  \eqref{eq:6termforExtwtTAn} and \eqref{eq:6termAnn+1} are 
isomorphic. 
\end{proof}
%%%%%%%%%%%%%%%%%%%%
%\begin{corollary}\label{cor:Exts0}
%\Exts^0(\wtT_{A_n}) \cong \Ker(I_n - \widetilde{A}_{C_n}).
%\end{corollary}
We consequently have the following theorem which will be used to prove 
Theorem \ref{thm:main2}.
\begin{theorem}\label{thm:6termforExtOA}
Assume that an infinite matrix $A$ is (DRS) and (LI).
Then the cyclic six term exact sequence \eqref{eq:6termExtforA} 
for the $\Ext$-groups of the Exel--Laca algebra $\OELA$ 
\begin{equation}\label{eq:6termforExtOA}
\begin{CD}
\Exts^0(\OELA) @>>> \Extw^0(\OELA) @>>>  \Z \\
@AAA @. @V{\iota_{\OELA}}VV \\
0 @<<<  \Extw^1(\OELA) @<<< \Exts^1(\OELA)
\end{CD}
\end{equation}
is computed to be the cyclic six term exact sequence
\eqref{eq:6termforExtwtTAn} and hence \eqref{eq:6termAnn+1} for $n \ge N$.
\end{theorem}
\begin{proof}
By Proposition \ref{prop:KformulaEL} and Proposition \ref{prop:mainExtsExtw1}, 
we have 
$\K_*(\OELA) \cong \K_*(\wtT_{A_n})$ and 
$\Extw^1(\OELA) \cong \Extw^1(\wtT_{A_n}),
\Exts^1(\OELA) \cong \Exts^1(\wtT_{A_n}).$
These isomorphisms are all induced by the natural embedding
$\wtT_{A_n} \hookrightarrow \OELA$.
By the UCT, we have the commutative diagram
\begin{equation*}
\begin{CD}
0 @>>> \Ext_\Z^1(\K_1(\OELA),\Z) @>>> \Extw^0(\OELA) @>>> \Hom(\K_0(\OELA),\Z) @>>> 0 \\
@. @VV{\varphi_{\K_1}}V                  @VV{\varphi_w^0}V   @VV{\varphi_{\K_0}}V  @. \\
0 @>>> \Ext_\Z^1(\K_1(\wtT_{A_n}),\Z) @>>> \Extw^0(\wtT_{A_n}) @>>> \Hom(\K_0(\wtT_{A_n}),\Z) @>>> 0.
\end{CD}
\end{equation*}
Since the two vertical arrows $\varphi_{\K_1}$ and $\varphi_{\K_0}$ are isomorphisms, 
the midle arrow 
$\varphi_w^0: \Extw^0(\OELA) \rightarrow \Extw^0(\wtT_{A_n})$
is an isomorphism.
In the commutative diagrams of the two six term exact sequences
\begin{equation*}
\begin{CD}
0 @>>> \Exts^0(\OELA) @>>> \Extw^0(\OELA) @>>> \Z \\ %@>>> \Exts^1(\OELA) @>>> \Extw^1(\OELA) @>> 0 \\
@. @VV{\varphi_s^0}V    @VV{\varphi_w^0}V  @| \\ % @VV{\varphi_s^1}V  @VV{\varphi_w^1}V  @. \\
0 @>>> \Exts^0(\wtT_{A_n}) @>>> \Extw^0(\wtT_{A_n}) @>>> \Z% @>>> \Exts^1(\wtT_{A_n}) @>>> 
%\Extw^1(\wtT_{A_n}) @>> 0 
\end{CD}
\end{equation*}
we have that $\varphi_s^0: \Exts^0(\OELA) \rightarrow \Exts^0(\wtT_{A_n})$
is an isomorphism, because 
$\varphi_w^0: \Extw^0(\OELA) \rightarrow \Extw^0(\wtT_{A_n})$ is an isomorphism.
All the isomorphisms between $\Ext_*^i(\OELA)$ and $\Ext_*^i(\wtT_{A_n})$
are induced by the natural embeddings
$\wtT_{A_n} \hookrightarrow \OELA$. 
Therefore we conclude that the six term exact sequence \eqref{eq:6termforExtOA}
is isomorphic to the cyclic six term exact sequence \eqref{eq:6termforExtwtTAn}.
\end{proof}

%%%%%%%%%%%%%%%%%%%%%%%%%%%%%%%%%%%%%%%%%%
\section{Main result}
%%%%%%%%%%%%%%%%%%%%%%%%%%%%%%%%%%%
Assume that an infinite matrix $A = [A(i,j)]_{i,j\in \N}$
is (RS) and (LI).
Let $K \in \N$ be the positive integer in Definition \ref{def:semifinite} (ii).
If $A$ further satisfies (DC), 
%Keep the notation as in the previous sections. By 
Theorem \ref{thm:6termforExtOA} holds so that
the strong extension groups $\Exts^1(\OELA), \Exts^0(\OELA)$
are computed to be $\Coker(I_n - \widetilde{A}_{C_n}), 
\Ker(I_n - \widetilde{A}_{C_n})$ 
for
$n \ge K$,
respectively.
As the finite matrix $\widetilde{A}_{C_n}$ defined by 
\eqref{eq:ACn} does not necessarily have its entries in nonnegative integers,
we will replace it with
another matrix having entries in $\{0, 1\},$
keeping the cokernel and kernel.

Fix $n \ge K$ and define the $(n+2) \times (n+2)$ matrix
$A_{C_{n}}$ and
 the $(n+2) \times 1$ matrix
 $\begin{bmatrix}
-C_n\\
0\\
0
\end{bmatrix}
$ by 
\begin{equation}\label{eq:wtAnCn}
A_{C_{n}}
:=
\begin{bmatrix}
  &       &   & c_1  & 0 \\
  & A_n   &   &\vdots& \vdots \\
  &       &   & c_n  & 0 \\
1 & \dots & 1 & 1    & 1 \\
0 & \dots & 0 & 1    & 0 
\end{bmatrix},
\qquad
\begin{bmatrix}
-C_n\\
0\\
0
\end{bmatrix}
=
\begin{bmatrix}
-c_1 \\
\vdots \\
-c_n \\
0 \\
0
\end{bmatrix}
\end{equation}
\begin{lemma} \label{lem:CokerKerIwAnCn}
Assume that the matrix $A$ is (RS) and (LI).
\begin{enumerate}
\renewcommand{\theenumi}{(\roman{enumi})}
\renewcommand{\labelenumi}{\textup{\theenumi}}
\item
$
(\Coker(I_n - \widetilde{A}_{C_n}), [-C_n])
\cong
(\Coker( I_{n+2} - A_{C_{n}}), 
\begin{bmatrix}
-C_n\\
0\\
0
\end{bmatrix}).
$
\item
$
\Ker(I_n - \widetilde{A}_{C_n})
\cong
\Ker( I_{n+2} - A_{C_{n}}).
$
\end{enumerate}
\end{lemma}
\begin{proof}
(i)
We have the following sequence of elementary operations on matrices:
{\allowdisplaybreaks
\begin{align*}
 & \left(\Cokern, [-C_n] \right) \\
%=
% & \left(\Coker
%\begin{bmatrix}
%1 - A(1,1) + c_1 & -A(1,2) +c_1 & \dots & -A(1,n) + c_1  \\
%  - A(2,1) + c_2 &1-A(2,2) +c_2 & \dots & -A(2,n) + c_2  \\
%\vdots           & \vdots       &\ddots & \vdots         \\
%  - A(n,1) + c_n & -A(n,2) +c_n & \dots &1-A(n,n) + c_n    
%\end{bmatrix},
%\begin{bmatrix}
%-c_1 \\
%-c_2 \\
%\vdots \\
%-c_n 
%\end{bmatrix}
%\right)
%\\
=
 & \left(\Coker
\begin{bmatrix}
1 - A(1,1) + c_1 & -A(1,2) +c_1 & \dots & -A(1,n) + c_1 & 0  \\
  - A(2,1) + c_2 &1-A(2,2) +c_2 & \dots & -A(2,n) + c_2 & 0 \\
\vdots           & \vdots       &\ddots & \vdots        & \vdots  \\
  - A(n,1) + c_n & -A(n,2) +c_n & \dots &1-A(n,n) + c_n & 0 \\   
            0    &   0          & \dots & 0             & -1
\end{bmatrix},
\begin{bmatrix}
-c_1 \\
-c_2 \\
\vdots \\
-c_n \\
0
\end{bmatrix}
\right)
 \\
=
 & \left(\Coker
\begin{bmatrix}
1 - A(1,1) + c_1 & -A(1,2) +c_1 & \dots & -A(1,n) + c_1 & -c_1  \\
  - A(2,1) + c_2 &1-A(2,2) +c_2 & \dots & -A(2,n) + c_2 & -c_2 \\
\vdots           & \vdots       & \ddots& \vdots        & \vdots  \\
  - A(n,1) + c_n & -A(n,2) +c_n & \dots &1-A(n,n) + c_n & -c_n \\   
            0    &   0          & \dots & 0             & -1
\end{bmatrix},
\begin{bmatrix}
-c_1, \\
-c_2 \\
\vdots \\
-c_n \\
0
\end{bmatrix}
\right)
 \\
%=
% &\left(\Coker
%\begin{bmatrix}
%1 - A(1,1)       & -A(1,2)      & \dots & -A(1,n)       & -c_1  \\
%  - A(2,1)       &1-A(2,2)      & \dots & -A(2,n)       & -c_2 \\
%\vdots           & \vdots       &\ddots & \vdots        & \vdots  \\
%  - A(n,1)       & -A(n,2)      & \dots &1-A(n,n)       & -c_n \\   
%        -1       &   -1         & \dots &   -1          & -1
%\end{bmatrix},
%\begin{bmatrix}
%-c_1 \\
%-c_2 \\
%\vdots \\
%-c_n \\
%0
%\end{bmatrix}
%\right)
% \\
=
& \left(\Coker  
\begin{bmatrix}
&         &    & -c_1 \\
&I_n -A_n &    & \vdots \\
&         &    & -c_n \\
-1 & \dots& -1 & -1 
\end{bmatrix},
\begin{bmatrix}
-c_1 \\
\vdots \\
-c_n \\
0
\end{bmatrix}
\right)
\\
=
& \left(\Coker  
\begin{bmatrix}
&         &    & -c_1   & 0 \\
&I_n -A_n &    & \vdots & \vdots \\
&         &    & -c_n   & 0 \\
-1 & \dots& -1 & -1     & 0 \\
 0 & \dots& 0  &  0     & 1   
\end{bmatrix},
\begin{bmatrix}
-c_1 \\
\vdots \\
-c_n \\
0 \\
0
\end{bmatrix}
\right)
\\
%=
%& \left(\Coker  
%\begin{bmatrix}
%&         &    & -c_1   & 0 \\
%&I_n -A_n &    & \vdots & \vdots \\
%&         &    & -c_n   & 0 \\
%-1 & \dots& -1 & -1     & -1\\
% 0 & \dots& 0  &  0     & 1   
%\end{bmatrix},
%\begin{bmatrix}
%-c_1 \\
%\vdots \\
%-c_n \\
%0 \\
%0
%\end{bmatrix}
%\right)
%\\
=
& \left(\Coker  
\begin{bmatrix}
&         &    & -c_1   & 0 \\
&I_n -A_n &    & \vdots & \vdots \\
&         &    & -c_n   & 0 \\
-1 & \dots& -1 & 0      & -1\\
 0 & \dots& 0  & -1     & 1   
\end{bmatrix},
\begin{bmatrix}
-c_1 \\
\vdots \\
-c_n \\
0 \\
0
\end{bmatrix}
\right)
\\
=
& \left(\Coker(  
I_{n+2} 
-
A_{C_n}),
%\begin{bmatrix}
%&         &    &  c_1   & 0 \\
%& A_n     &    & \vdots & \vdots \\
%&         &    &  c_n   & 0 \\
% 1 & \dots&  1 &  1     &  1\\
% 0 & \dots& 0  &  1     & 0   
%\end{bmatrix},
\begin{bmatrix}
-C_n \\
0 \\
0
\end{bmatrix}
\right)
\end{align*}
}
(ii)
The above sequence of matrices  
shows $
\Ker(I_n - \widetilde{A}_{C_n})
\cong
\Ker( I_{n+2} - A_{C_{n}}).
$
\end{proof}
We will next replace the matrix $A_{C_n}$ with
the following $(n+2)\times (n+2)$-matrix $\widehat{A}_{C_n}$ 
defined by
\begin{equation}\label{eq:matrixhatACn}
\widehat{A}_{C_n}
:= 
\begin{bmatrix}
    &       &    & 1 & 0 \\
    & A_n^t &    & \vdots & \vdots \\
    &       &    & 1 & 0 \\
   1& \dots & 1  & 1 & 1 \\
c_1 & \dots & c_n& 1 & 1 
\end{bmatrix}
\end{equation}
where 
$A_n =[A(i,j)]_{i,j=1}^n,$
so that we will know that 
$$
\left(
\Coker(I_{n+2}-A_{C_n}), 
\begin{bmatrix}
-C_n \\
0 \\
0
\end{bmatrix}
\right)
=
\left(
\Coker(I_{n+2}-(\widehat{A}_{C_n})^t), 
\begin{bmatrix}
1 \\
\vdots \\
1
\end{bmatrix}
\right).
$$
The matrix 
$\widehat{A}_{C_n}$ is irreducible, non-permutation
because $A$ satisfies (LI),
so that Cuntz--Krieger algebra 
$\OCK_{\widetilde{A}_{C_n}}$
for the finite matrix
$\widetilde{A}_{C_n}$
is simple and pure infinite.
\begin{proposition}\label{prop:main1}
Let $A = [A(i,j)]_{i,j\in \N}$
be an irreducible infinite matrix with entries in $\{0,1\}$.
 Assume that $A$ is (RS) and (LI).
 Then we have for $n \ge K$
\begin{align*}
(\Exts^1(\wtT_{A_n}), [\iota_{A_n}(1)]_s, \Exts^0(\wtT_{A_n}))
\cong
(\K_0(\OCK_{\widehat{A}_{C_n}}), [1_{\OCK_{\widehat{A}_{C_n}}}]_0, \K_1(\OCK_{\widehat{A}_{C_n}})). 
\end{align*}
\end{proposition}
\begin{proof}
Since
\begin{align*}
\begin{bmatrix}
&         &    & -c_1   & 0 \\
&I_n -A_n &    & \vdots & \vdots \\
&         &    & -c_n   & 0 \\
-1 & \dots& -1 & 0      & -1\\
 0 & \dots& 0  & -1     & 1   
\end{bmatrix}
\begin{bmatrix}
0\\
\vdots \\
0 \\
-1 \\
-1 
\end{bmatrix}
=
\begin{bmatrix}
c_1\\
\vdots \\
c_n \\
1\\
0 
\end{bmatrix},
\end{align*}
two vectors
$
\begin{bmatrix}
0\\
\vdots \\
0 \\
1 \\
0 
\end{bmatrix}
$
and
$
\begin{bmatrix}
-c_1\\
\vdots \\
-c_n \\
0\\
0 
\end{bmatrix}
$
define the same element in 
the group
$\Coker(I_{n+2} - A_{C_n})$.
By the elementary operations on matrices, we have
{\allowdisplaybreaks
\begin{align*}
& \left(\Coker(  
I_{n+2} 
-
A_{C_n}),
\begin{bmatrix}
-C_n \\
0 \\
0
\end{bmatrix}
\right) \\
=
& \left(
\Coker
{\begin{bmatrix}
&         &    & -c_1   & 0 \\
&I_n -A_n &    & \vdots & \vdots \\
&         &    & -c_n   & 0 \\
-1 & \dots& -1 & 0      & -1\\
 0 & \dots& 0  & -1     & 1   
\end{bmatrix},
} \, 
\begin{bmatrix}
0\\
\vdots \\
0 \\
1 \\
0 
\end{bmatrix}
\right)
\\
= 
& \left(
\Coker
{\begin{bmatrix}
&         &    & -c_1   & 0 \\
&I_n -A_n &    & \vdots & \vdots \\
&         &    & -c_n   & 0 \\
-1 & \dots& -1 & 0      & -1\\
-1 & \dots& -1 & -1     & 0   
\end{bmatrix},
} \, 
\begin{bmatrix}
0\\
\vdots \\
0 \\
1 \\
1 
\end{bmatrix}
\right)
\\
= 
& \left(
\Coker
{\begin{bmatrix}
&         &    & -c_1   & 0 \\
&I_n -A_n &    & \vdots & \vdots \\
&         &    & -c_n   & 0 \\
0  & \dots& 0  & 0      & -1\\
-1 & \dots& -1 & -1     & 0   
\end{bmatrix},
} \, 
\begin{bmatrix}
0\\
\vdots \\
0 \\
1 \\
1 
\end{bmatrix}
\right)
\\
= 
& \left(
\Coker
{\begin{bmatrix}
&         &    & -c_1   & -1 \\
&I_n -A_n &    & \vdots & \vdots \\
&         &    & -c_n   & -1 \\
0  & \dots& 0  & 0      & -1 \\
-1 & \dots& -1 & -1     & 0   
\end{bmatrix},
} \, 
\begin{bmatrix}
1\\
\vdots \\
1 \\
1 \\
1 
\end{bmatrix}
\right)
\\
= 
&
\left(
\Coker(
I_{n+2} -
{\begin{bmatrix}
&         &    &  c_1   &  1 \\
&A_n &    & \vdots & \vdots \\
&         &    &  c_n   &  1 \\
0  & \dots& 0  & 1      &  1 \\
1 & \dots&  1 &  1     & 1   
\end{bmatrix}),
} \, 
\begin{bmatrix}
1\\
\vdots \\
1 \\
1 \\
1 
\end{bmatrix}
\right)
\\
= 
&\left(
\Coker(
I_{n+2} -
{\begin{bmatrix}
&         &    &    1   & c_1\\
&A_n &    & \vdots & \vdots \\
&         &    &   1    &c_n \\
1  & \dots& 1  & 1      &  1 \\
0 & \dots&  0 &  1     & 1   
\end{bmatrix}),
} \, 
\begin{bmatrix}
1\\
\vdots \\
1 \\
1 \\
1 
\end{bmatrix}
\right)
\\
= &
\left(
\Coker( I_{n+2} - (\widehat{A}_{C_n})^t), \, 
\begin{bmatrix} 
1\\
\vdots \\
1 \\
1 \\
1 
\end{bmatrix}
\right).
\end{align*}
}
The above matrix operations 
show us 
$\Ker(I_{n+2} - A_{C_n}) \cong \Ker( I_{n+2} - (\widehat{A}_{C_n})^t)
$ which is isomorphic to $\Exts^0(\wtT_{A_n})$
by Lemma \ref{lem:KerwAnn+1} ,Proposition \ref{prop:6termforExtwtTAn}
and Lemma \ref{lem:CokerKerIwAnCn}.
\end{proof}

Recall that the algebra $\widehat{\OELA}$ is the unital Kirchberg algebra satisfying
\begin{align*}
(K_0(\widehat{\OELA}), [1_{\widehat{\OELA}}]_0, K_1(\widehat{\OELA}))
\cong & (\Exts(\OELA), \iota_{\OELA}(1), \Exts^0(\OELA)) \\
= &(\Exts(\wtT_{A_n}), \iota_{A_n}(1), \Exts^0(\wtT_{A_n})).
\end{align*}
Therefore Proposition \ref{prop:main1} yields the following.
\begin{theorem}\label{thm:main2}
Let $A = [A(i,j)]_{i,j\in \N}$
be an irreducible infinite matrix with entries in $\{0,1\}$.
 Assume that $A$ is (DRS) and (LI).
Then there exists $K \in \N$ such that the reciprocal dual 
$\widehat{\OELA}$ of the Exel--Laca algebra 
$\OELA$ is isomorphic to the simple Cuntz--Krieger algebra 
$\OCK_{\widehat{A}_{C_K}}$ for the $(K+2)\times (K+2)$-finite matrix 
$\widehat{A}_{C_K}$ defined by \eqref{eq:matrixhatACK}.
\end{theorem}
By excahnging the $(K+1)$-row (resp. column) with
 $(K+2)$-row (resp. column) simultaneously in the matrix \eqref{eq:matrixhatACK},
 we have the following corollary.
\begin{corollary}\label{cor:main2}
Keep the assumption of Theorem \ref{thm:main2}.
Let
$\widehat{A}_{[K,K+1]}$
be the $(K+2)\times (K+2)$ matrix defined by
\begin{align*}
\widehat{A}_{[K,K+1]}
:= 
{\left[
\begin{array}{ccccc}
A(1,1)  & \dots & A(K,1) &  0   & 1 \\ 
A(1,2)  & \dots & A(K,2) &  0   &1 \\ 
\vdots  &       & \vdots &\vdots& \vdots \\ 
A(1,K)  & \dots & A(K,K) &  0   & 1 \\ 
A(1,K+1)& \dots &A(K,K+1)&  1   & 1 \\ 
   1    & \dots &  1     &  1   & 1 
\end{array}
\right] } 
= 
\begin{bmatrix}
   &       &    & 0      & 1 \\
   & A_K^t &    & \vdots &\vdots \\
   &       &    & 0      & 1 \\
c_1&\dots  &c_K & 1      & 1 \\
 1 &\dots  & 1  & 1      & 1 
 \end{bmatrix}.
\end{align*}
Then we have
$\widehat{\OELA} = \OCK_{\widehat{A}_{[K,K+1]}}$.
\end{corollary}

%%%%%%%%%%%%%%%%%%%%%%%%%%%%%%%%%%%%%%%%%%%%%%%%%%%%%%%%%%%%%%%%%%%%%%%%%%%%%%%%%
\section{Examples of the reciprocal of Exel--Laca algebras}\label{sect:Examples}
%%%%%%%%%%%%%%%%%%%%%%%%%%%%%%%%%%%%%%%%%%

{\bf 1.} The reciprocal dual of $\OI$:

$A=[A(i,j)]_{i,j\in \N}$  be the infinite matrix  $I_\infty$
defined by \eqref{eq:Iinftymatrix}.
%whose entries are all $1'$ s,
%such as 
%\begin{equation}\label{eq:Iinftymatrix}
%I_\infty =
%\begin{bmatrix}
%1 & 1 & 1 &\cdots &   & \\
%1 & 1 & 1 &\cdots &   & \\
%1 & 1 & 1 &\cdots &   & \\
%1 & 1 & 1 & 1 & 1 & 1 &\cdots &   & \\
%1 & 1 & 1 & 1 & 1 & 1 &       &   & \\
%\vdots &\vdots &\vdots&\ddots &   & \\            
%  &       &      &       &   &              
%\end{bmatrix}
%\end{equation}
All the vertices are row complementary finite, so that 
$c_i =1$ for all $ i \in \N$.
Hence for any $n \ge 1$, we have $A(i,n+1) =1 = c_i$ for $i=1,\dots,n+1.$
It is (DRS) and (LI) with $K=1$.
Its Exel--Laca algebra 
$\OEL_{I_\infty}$ is the Cuntz algebra $\OI$ satisfying
$\K_0(\OI) =\Z, \K_1(\OI) =0.$

Take $n$ =1, so that we have
$$
\widehat{A}_{C_1}
=
\begin{bmatrix}
A_1^t & 1 & 0 \\
1     & 1 & 1 \\
0     & 1 & 0
\end{bmatrix}
$$
By elementary operations on the matrices, we have
\begin{equation*}
I_3 - \widehat{A}_{C_1}^t 
=
\begin{bmatrix}
0 &-1 & 0 \\
-1& 0 & -1\\
0 & -1& 1
\end{bmatrix}
\longrightarrow
\begin{bmatrix}
0 &-1 & 0 \\
-1& 0 & 0\\
0 & -1& 1
\end{bmatrix}
\longrightarrow
\begin{bmatrix}
0 &-1 & 0 \\
-1& 0 & 0\\
0 & 0 & 1
\end{bmatrix}
\end{equation*}
Therefore we have
\begin{align*}
\Exts^1(\OI) = \Exts^1(\OEL_{I_\infty})
=\Coker 
(I_{3} - \widehat{A}_{C_1}^t)
= 0
\end{align*}
showing that 
$\K_0(\widehat{\O}_\infty) = \Exts^1(\OI) =0$
and similarly
$\K_1(\widehat{\O}_\infty) = \Exts^0(\OI) =0,$
and hence 
$\widehat{\O}_\infty = \O_2$.

%%%%%%%%%%%%%%%%%%%%%%%%%%%%%%%%%%%%%%%%
{\bf 2.} The reciprocal duals of Cuntz--Krieger algebras:

Let $A=[A(i,j)]_{i,j=1}^N$ be an irreducible non-permutation matrix with entries in $\{0,1\}.$
the reciprocal dual $\widehat{\OCKA}$ of the Cuntz--Krieger algebra $\OCKA$ is 
a unital Kirchberg algebra which is  the Exel--Laca algebra
$\OELwA$  
for the  infinite matrix $\widehat{A}$ defined by \eqref{eq:widehatA}.
We put $B =\widehat{A}$.
%%:
%\begin{equation}\label{eq:widehatA}
%\widehat{A}=
%\left[
%\begin{array}{ccccccc}
%&                &1      &0     &0     &\cdots \\
%&\text{\Huge A}^t&\vdots&\vdots&\vdots&\vdots \\
%&                &1      &0     &0     &\cdots \\
%1&\dots          &1      &1     &1     &\cdots \\
%1&\dots          &1      &0     &0     &\cdots \\
%1&\dots          &1      &0     &0     &\cdots \\
%\vdots &\dots    &\vdots &\vdots&\vdots&\vdots \\
%\end{array}
%\right]
%\end{equation}
The vertex $N+1\in V_B$ is row complementary finite, the other vertices
are all row finite. 
It is (DRS) and (LI) with $K =N+1$.
We then have
$B(i,n+1) =c_i$ for all $i=1,\dots, n+1$ for $n \ge K$.
The $K \times K$ matrix $B_K=[B(i,j)]_{i,j=1}^K$  is 
\begin{equation}\label{eq:BK}
B_K=
\left[
\begin{array}{ccccccc}
&                &1         \\
&\text{\Huge A}^t&\vdots  \\
&                &1            \\
1&\dots          &1              
\end{array}
\right] 
\quad
\text{ with }
\quad
\begin{bmatrix}
c_1\\
\vdots \\
c_{K-1} \\
c_{K}
\end{bmatrix}
=
\begin{bmatrix}
0\\
\vdots \\
0 \\
1 
\end{bmatrix}.
\end{equation}
The matrix $\widehat{B}_{C_K}$ defined by 
\eqref{eq:matrixhatACK} for $A =B$ is the $(K+2)\times (K+2)$ matrix such that 
\begin{equation}\label{eq:tildeBhatK}
\widehat{B}_{C_K}
=
\begin{bmatrix}
  &                &  &     & 1      & 0 \\
& & ({B_K})^t      &  &    \vdots & \vdots \\ 
  &                &  &     & 1      & 0 \\
  &                &  &     & 1      & 0 \\
1 & \dots          & 1& 1   & 1      & 1 \\
c_1 & \dots        &c_{K-1}&c_{K}    & 1      & 1 
\end{bmatrix}
=
\begin{bmatrix}
  &                &  & 1   & 1      & 0 \\
  & A              &  & \vdots       & \vdots & \vdots \\ 
  &                &  & 1   & 1      & 0 \\
1 & \dots          & 1& 1   & 1      & 0      \\
1 & \dots          & 1& 1   & 1      & 1 \\
0 & \dots          & 0& 1   & 1      & 1 
\end{bmatrix}.
\end{equation}
By the elementary operations on matrices, we have 
{\allowdisplaybreaks
\begin{align*}
& (
\K_0(\OCK_{\widehat{B}_{C_K}}), [1_{\OCK_{\widehat{B}_{C_K}}}]_0 ) \\
= &  \left(
(\Coker( I_{N+3} - ({\widehat{B}_{C_K}})^t), \,
\begin{bmatrix}
1\\
\vdots \\
1 
\end{bmatrix}
\right)
\\
= &
\left(
\Coker 
\begin{bmatrix}
  &                &  & -1            &-1      & 0 \\
  &I_N- A^t        &  & \vdots        & \vdots & \vdots \\ 
  &                &  & -1            &-1      & 0 \\
-1& \dots          &-1& 0             &-1      & -1 \\
-1& \dots          &-1&-1             & 0      &-1 \\
0 & \dots          & 0& 0             & -1     & 0 
\end{bmatrix}, \,
\begin{bmatrix}
1\\
\vdots \\
1 \\
1 \\
1 \\
1 
\end{bmatrix}
\right) 
\\
= &
\left(
\Coker 
\begin{bmatrix}
  &                &  & -1            & 0      & 0 \\
  &I_N- A^t        &  & \vdots        & \vdots & \vdots \\ 
  &                &  & -1            & 0      & 0 \\
0 & \dots          &0 & 0             &-1      & -1 \\
0 & \dots          &0 &-1             & 1      &-1 \\
0 & \dots          & 0& 0             & -1     & 0 
\end{bmatrix}, \,
\begin{bmatrix}
1\\
\vdots \\
1 \\
1 \\
1 \\
1 
\end{bmatrix}
\right) \\
= &
\left(
\Coker 
\begin{bmatrix}
  &                &  & -1            & 0      & 0 \\
  &I_N- A^t        &  & \vdots        & \vdots & \vdots \\ 
  &                &  & -1            & 0      & 0 \\
0 & \dots          &0 & 0             &-1      & -1 \\
0 & \dots          &0 &-1             & 2      & 0 \\
0 & \dots          & 0& 0             & -1     & 0 
\end{bmatrix}, \,
\begin{bmatrix}
1\\
\vdots \\
1 \\
1 \\
0 \\
1 
\end{bmatrix}
\right) \\
= &
\left(
\Coker 
\begin{bmatrix}
  &                &  & -1            & 0      & 0 \\
  &I_N- A^t        &  & \vdots        & \vdots & \vdots \\ 
  &                &  & -1            & 0      & 0 \\
0 & \dots          &0 & 0             &-1      & -1 \\
0 & \dots          &0 &-1             & 0      & 0 \\
0 & \dots          & 0& 0             & -1     & 0 
\end{bmatrix}, \,
\begin{bmatrix}
1\\
\vdots \\
1 \\
1 \\
2 \\
1 
\end{bmatrix}
\right) \\
= &
\left(
\Coker 
\begin{bmatrix}
  &                &  & -1            & 0      & 0 \\
  &I_N- A^t        &  & \vdots        & \vdots & \vdots \\ 
  &                &  & -1            & 0      & 0 \\
0 & \dots          &0 & 0             & 0      & -1 \\
0 & \dots          &0 &-1             & 0      & 0 \\
0 & \dots          & 0& 0             & -1     & 0 
\end{bmatrix}, \,
\begin{bmatrix}
-1\\
\vdots \\
-1 \\
1 \\
2 \\
1 
\end{bmatrix}
\right) \\
=&
\left(
\Coker
\begin{bmatrix}
  &                &   \\
  &I_N- A^t        &   \\ 
  &                &   \\
\end{bmatrix}, \,
\begin{bmatrix}
-1\\
\vdots \\
-1 \\
\end{bmatrix}
\right) \\
\cong &
(\K_0(\OCKA), [1_{\OCKA}]_0)
\end{align*}
}
Hence we have 
$\OCK_{\widehat{B}_{C_K}} \cong \OCKA$ 
and hence
$\widehat{(\widehat{\OCKA})} =\OCKA.$

 We remark that the equality
$\det(I_{N+3} - \widehat{B}_{C_K}) = - \det(I_N- A)$
holds by a direct computation.

%%%%%%%%%%%%%%%%%%%%%%%%%%%%%%%%%%%%%%%%%%%

{\bf 3.} The Kirchberg algebra $\PI$:

Let $A$ be  the infinite matrix $P_\infty$ defined by \eqref{eq:Pinftymatrix}.
%such as
%\begin{equation}\label{eq:Pinftymatrix} 
%P_\infty =
%\begin{bmatrix}
%1 & 0 & 1 & 1 & 1 & 1 &\cdots &   & \\
%0 & 1 & 1 & 1 & 1 & 1 &\cdots &   & \\
%1 & 0 & 1 & 0 & 0 & 0 &\cdots &   & \\
%0 & 1 & 0 & 1 & 0 & 0 &\cdots &   & \\
%0 & 0 & 1 & 0 & 1 & 0 &       &   & \\
%  &\ddots &\ddots &\ddots   &\ddots &\ddots&\ddots &   & \\            
%  &   &   &   &   & & &   &             
%\end{bmatrix}
%\quad (\text{cf. } \cite[Example \, 4.2]{RaebSzy})
%\end{equation}
As in Section \ref{sec:Preliminaries},
 $P_\infty$ is (RS) and (LI).
 Its Exel--Laca algebra $\OEL_{P_\infty}$ is the Kirchberg algebra
 $\PI$ satisfying $\K_0(\PI) = 0, \K_1(\PI) =\Z$.
As
$c_1 = c_2 =1$ and $c_i =0$ for $i\ge 3$,
and $A(n+1, n+1) =1, c_{n+1} =0$
 for $n \ge 2$, 
  the matrix $A=P_\infty$ is (RS) and (LI) with $K =2$,
  but not (DRS).
Let us compute $\Exts^i(\wtT_{A_n}), i=0,1$ 
by using Proposition \ref{prop:6termforExtwtTAn}
with \eqref{eq:6termAnn+1}.
We fix $n \ge 4$.
We then have
\begin{align*}
\Coker(I_n - \widetilde{A}_{C_n})
%& =  
%\Coker
%{\begin{bmatrix}
%1 - A(1,1) + c_1 & -A(1,2) +c_1 & \dots & -A(1,n) + c_1  \\
%  - A(2,1) + c_2 &1-A(2,2) +c_2 & \dots & -A(2,n) + c_2  \\
%\vdots           & \vdots       &       & \vdots         \\
%  - A(n,1) + c_n & -A(n,2) +c_n & \dots &1-A(n,n) + c_n    
%\end{bmatrix}
%} \\
& = 
\Coker
{\begin{bmatrix}
1     & 1    & 0    &    0 &\dots & 0 \\
1     & 1    & 0    &    0 &\dots & 0 \\
-1    & 0    & 0    &    0 &\dots & 0 \\
 0    &-1    & 0    &    0 &\dots & 0 \\
\vdots&\ddots&\ddots&\ddots&\ddots& \vdots  \\
 0    & \dots& 0    &-1    & 0    & 0
\end{bmatrix}
} \\
& = 
\Coker
{\begin{bmatrix}
0     & 0    & 0    &    0 &\dots & 0 \\
0     & 0    & 0    &    0 &\dots & 0 \\
-1    & 0    & 0    &    0 &\dots & 0 \\
 0    &-1    & 0    &    0 &\dots & 0 \\
\vdots&\ddots&\ddots&\ddots&\ddots& \vdots  \\
 0    & \dots& 0    &-1    & 0    & 0
\end{bmatrix}
} 
 =
\Z
{\begin{bmatrix}
1 \\
0 \\
0 \\
0 \\
 \vdots  \\
 0
\end{bmatrix}
}
\oplus
\Z
{\begin{bmatrix}
0 \\
1 \\
0 \\
0 \\
 \vdots  \\
 0
\end{bmatrix}
} \\
& 
=  \Z e_1 \oplus \Z e_2. 
 \end{align*}
Since $I_{n+1,n}:\Coker(I_{n+1} - \widetilde{A}_{C_{n+1}}) 
\rightarrow \Coker(I_n - \widetilde{A}_{C_n})
$
is defined by $I_{n+1,n}([k_i]_{i=1}^{n+1}) =  [k_i]_{i=1}^{n}$,
by the above identification 
$\eta_n: \Coker(I_n - \widetilde{A}_{C_n}) \rightarrow \Z e_1 \oplus \Z e_2$,
we have a commutative diagram
\begin{equation*}
\begin{CD}
\Coker(I_{n+1} - \widetilde{A}_{C_{n+1}}) @>{I_{n+1,n}}>> \Coker(I_n - \widetilde{A}_{C_n}) \\
@V{\eta_{n+1}}VV    @V{\eta_{n}}VV \\    
\Z e_1 \oplus \Z e_2@>{\id}>> \Z e_1\oplus \Z e_2
\end{CD}.
\end{equation*}
Since $\PI$ has finitely generated $\K$-groups, 
the proof of 
 Proposition \ref{prop:mainExtsExtw1} tells us that 
 $\Exts^1(\OEL_{P_\infty})
\cong \varprojlim \Exts^1(\wtT_{A_n})
$ so that 
 \[
\Exts^1(\calP_\infty) 
\cong \Exts^1(\OEL_{P_\infty})
\cong \varprojlim \Coker(I_n - \widetilde{A}_{C_n}) 
\cong \Z \oplus \Z.
\]
Similarly we have 
\begin{align*}
\Ker(I_n - \widetilde{A}_{C_n})
%& =  
%\Ker(
%{\begin{bmatrix}
%1 - A(1,1) + c_1 & -A(1,2) +c_1 & \dots & -A(1,n) + c_1 & 0 \\
%  - A(2,1) + c_2 &1-A(2,2) +c_2 & \dots & -A(2,n) + c_2 & 0 \\
%\vdots           & \vdots       &       & \vdots        & \vdots \\
%  - A(n,1) + c_n & -A(n,2) +c_n & \dots &1-A(n,n) + c_n & 0   
%\end{bmatrix}
%} )\\
& = 
\Ker
{\begin{bmatrix}
1     & 1    & 0    &    0 &\dots & 0 \\
1     & 1    & 0    &    0 &\dots & 0 \\
-1    & 0    & 0    &    0 &\dots & 0 \\
 0    &-1    & 0    &    0 &\dots & 0 \\
\vdots&\ddots&\ddots&\ddots&\ddots& \vdots  \\
 0    & \dots& 0    &-1    & 0    & 0
\end{bmatrix}
}
%& = 
%\Coker(
%{\begin{bmatrix}
%0     & 0    & 0    &    0 &\dots & 0 \\
%0     & 0    & 0    &    0 &\dots & 0 \\
%-1    & 0    & 0    &    0 &\dots & 0 \\
% 0    &-1    & 0    &    0 &\dots & 0 \\
%\vdots&\ddots&\ddots&\ddots&\ddots& \vdots  \\
% 0    & \dots& 0    &-1    & 0    & 0
%\end{bmatrix}
%}) \\
 =
\Z
{\begin{bmatrix}
0 \\
0 \\
 \vdots  \\
0\\
1\\
0
\end{bmatrix}
}
\oplus
\Z
{\begin{bmatrix}
0 \\
0 \\
 \vdots  \\
0\\
0\\
1
\end{bmatrix}
} \\
& 
= \Z e_{n-1}\oplus \Z e_n. 
 \end{align*}
Let $L: \Z \oplus \Z \rightarrow \Z \oplus \Z$ be the map defined by 
$L(k,l) = (0,k)$.
By the above identification 
$\zeta_n: \Ker(I_n - \widetilde{A}_{C_n}) \rightarrow \Z e_{n-1} \oplus \Z e_n$,
we have a commutative diagram
\begin{equation*}
\begin{CD}
\Ker(I_{n+1} - \widetilde{A}_{C_{n+1}}) @>{I_{n+1,n}}>> \Ker(I_n - \widetilde{A}_{C_n}) \\
@V{\zeta_{n+1}}VV    @V{\zeta_{n}}VV \\    
\Z e_{n-1} \oplus \Z e_n @>{L}>> \Z e_{n-1} \oplus \Z e_n
\end{CD}.
\end{equation*}
so that %As in the proof of Proposition \ref{prop:mainExtsExtw1} (i),
\[
\Exts^0(\calP_\infty) 
\cong \Exts^0(\OEL_{P_\infty})
\cong \varprojlim \Ker(I_n - \widetilde{A}_{C_n}) 
\cong 0,
\]
because of the last projective limit above  is taken along the map 
$L: \Z\oplus \Z \rightarrow \Z \oplus \Z$.
Since $\K_0(\widehat{\calP}_\infty) = \Exts^1(\calP_\infty) \cong \Z\oplus \Z$
and $\K_1(\widehat{\calP}_\infty) = \Exts^0(\calP_\infty) \cong 0,$
the torsion free part of $\K_0(\widehat{\calP}_\infty)$
is not isomorphic to $\K_1(\widehat{\calP}_\infty),$
the reciprocal dual $\widehat{\calP}_\infty$ of $\calP_\infty$
can not be realized as a simple Cuntz--Krieger algebra.

%%%%%%%%%%%%%%%%%%%%%%%%%%%%%
\section{Isomorphism invariant of simple Cuntz--Krieger algebras}
%%%%%%%%%%%%%%%%%%%%%%%%%%%%%%%%%%%%%%%%%%
Let $A=[A(i,j)]_{i,j=1}^N$ be an $N \times N$ irreducible non-permutation matrix
with entries in $\{0,1\}.$ 
The reciprocal dual $\widehat{\OCKA}$ of the Cuntz--Krieger algebra $\OCKA$ is 
a unital Exel--Laca algebra
$\OELwA$ for the  infinite matrix $\widehat{A}$
defined by \eqref{eq:widehatA}.
%%%%%%%%%%%%%%%%%%%%%
%\begin{equation}\label{eq:widehatA}
%\widehat{A}=
%\left[
%\begin{array}{ccccccc}
%&                &1      &0     &0     &\cdots \\
%&\text{\Huge A}^t&\vdots&\vdots&\vdots&\vdots \\
%&                &1      &0     &0     &\cdots \\
%1&\dots          &1      &1     &1     &\cdots \\
%1&\dots          &1      &0     &0     &\cdots \\
%1&\dots          &1      &0     &0     &\cdots \\
%\vdots &\dots    &\vdots &\vdots&\vdots&\vdots \\
%\end{array}
%\right].
%\end{equation}
As in the previous section, the matrix $\widehat{A}$ is (DRS) and (LI)
with $K =N+1$.
\begin{lemma} 
\begin{enumerate}
\renewcommand{\theenumi}{(\roman{enumi})}
\renewcommand{\labelenumi}{\textup{\theenumi}}
\item
$\Exts^1(\widehat{\OCKA}) %\cong \Exts^1(\wtT_{\widehat{A}_n})
\cong \Z^N/(I_N - A^t)\Z^N.$
\item
$\Extw^1(\widehat{\OCKA}) %\cong \Extw^1(\wtT_{\widehat{A}_n})
\cong \Z^N/
\begin{bmatrix}
&     & 1 \\
& I_N - A^t & \vdots \\
&     & 1
\end{bmatrix} \Z^{N+1}
$
\end{enumerate}
\end{lemma}
\begin{proof}
(i)
By the reciprocal duality, we have 
$\Exts^1(\widehat{\OCKA}) \cong \K_0(\OCKA)$,
so that 
$\Exts^1(\widehat{\OCKA})\cong \Z^N/(I_N - A^t)\Z^N$.

(ii)
Since
$\Extw^1(\widehat{\OCKA}) 
= \Extw^1(\wtT_{\widehat{A}_n})$,
it suffices to show  that 
$\Extw^1(\wtT_{\widehat{A}_n})\cong \Z^N/
\begin{bmatrix}
&     & 1 \\
& I_N - A^t & \vdots \\
&     & 1
\end{bmatrix} \Z^{N+1}
$ for $n=N+2$.
Since 
$c_i =0$ for $i \ne N+1$ and $c_{N+1} =1$, 
we have
{\allowdisplaybreaks
\begin{align*}
\Extw^1(\wtT_{\widehat{A}_{N+2}})
= & \Coker([I_n -\widehat{A}_n | -C_n]: \Z^{N+3}\rightarrow \Z^{N+2})\\
= & \Coker
\begin{bmatrix}
  &          &  & -1     & 0       & -c_1 \\ 
  &I_N - A^t &  & \vdots & \vdots & \vdots \\
  &          &  & -1     & 0       & -c_N \\ 
-1&\dots     &-1& 0      & -1      & -c_{N+1} \\
-1&\dots     &-1& -1     & 1      & -c_{N+2} 
\end{bmatrix}
\\
= & \Coker
\begin{bmatrix}
  &          &  & -1     & 0       & 0 \\ 
  &I_N - A^t &  & \vdots & \vdots & \vdots \\
  &          &  & -1     & 0       & 0    \\ 
-1&\dots     &-1& 0      & -1      & -1       \\
-1&\dots     &-1& -1     & 1      & 0        
\end{bmatrix}
\\
= & \Coker
\begin{bmatrix}
  &          &  & -1     & 0       & 0 \\ 
  &I_N - A^t &  & \vdots & \vdots & \vdots \\
  &          &  & -1     & 0       & 0    \\ 
0 &\dots     &0 & 0      & 0       & -1       \\
0 &\dots     &0 & 0      & 1      & 0        
\end{bmatrix}
\\
= & \Coker
\begin{bmatrix}
  &          &  & 1     \\ 
  &I_N - A^t &  & \vdots \\
  &          &  & 1       
\end{bmatrix}.
\end{align*}
}
\end{proof}
\begin{proposition}\label{prop:InvariantforCK}
For an $N \times N$ irreducible non-permutation matrix with entries in $\{0,1\},$
the pair of the  two groups
\begin{equation}
(\Z^N/(I_n - A^t)\Z^N, \, \, \Z^N/
\begin{bmatrix}
&     & 1 \\
& I_N - A^t & \vdots \\
&     & 1
\end{bmatrix} \Z^{N+1})
\end{equation}
is a complete invariant of the isomorphism class of the simple Cuntz--Krieger algebra $\OCKA$.
\end{proposition}
\begin{proof}
By the reciprocal duality,  we have 
$
\Exts^1(\widehat{\OCKA}) \cong \K_0(\OCKA), 
\Extw^1(\widehat{\OCKA}) \cong \K_0(\OCKA)/\Z[1_{\OCKA}]_0.
$
By \cite[Proposition 2.19]{SogabeJFA}, 
the pair $(\K_0(\OCKA), \K_0(\OCKA)/\Z[1_{\OCKA}]_0)$ 
is a complete list of invariants of the isomorphism class of
$\OCKA$, 
so that 
the pair $\Exts^1(\widehat{\OCKA}), \Extw^1(\widehat{\OCKA})$ 
is a complete invariant of $\OCKA$.
\end{proof}
By using \cite[Proposition 5.5]{MatSogabe2}, one knows that 
the direct sum 
\begin{equation} \label{eq:directsum}
\Z^N/(I_n - A^t)\Z^N 
\oplus 
\Z^N/ 
\begin{bmatrix}
%\left[
%\begin{array}{c|c|c}
&     & 1 \\
& I_n - A^t & \vdots \\
&     & 1
%\end{array}
%\right]
\end{bmatrix} 
\Z^{N+1}
\end{equation}
of the two abelian groups remenbers its direct summands.
Hence we have
\begin{corollary}\label{cor:InvariantforCK}
Let $A$ be an $N \times N$ irreducible non permutation matrix with entries in $\{0, 1\}.$
Then the  abelian group 
\begin{equation*}
\Z^{2N}/
%begin{bmatrix}
\left[
\begin{array}{c|c|c}
           &           & 0 \\
I_n - A^t  &   0       & \vdots \\
           &           & 0      \\
	   \hline
           &           & 1      \\
   0       & I_n - A^t & \vdots \\
           &           & 1
\end{array}
\right]
%\end{bmatrix} 
\Z^{2N+1}
\end{equation*}
 is a complete invariant of the isomorphism class 
of the simple Cuntz--Krieger algebra $\OCKA$.
\end{corollary}
We finally give a pair of examples to apply
Proposition \ref{prop:InvariantforCK}. 
Let $A, B$ be the two irreducible non-permutation matrices with entries in $\{0, 1\}$
such that
\begin{equation*}
A =
\begin{bmatrix}
1 & 1 & 1 \\
1 & 1 & 1 \\
1 & 0 & 0 \\ 
\end{bmatrix}, \qquad
B = A^t =
\begin{bmatrix}
1 & 1 & 1 \\
1 & 1 & 0 \\
1 & 1 & 0 \\ 
\end{bmatrix}.
\end{equation*}
It is direct to see that 
\begin{align*}
(\Z^3/(I_3 - A^t)\Z^3, \, \, \Z^3/
{\begin{bmatrix}
&     & 1 \\
& I_3 - A^t & \vdots \\
&     & 1
\end{bmatrix} \Z^{4})
}
& = (\Z / 2\Z, 0),\\ 
(\Z^3/(I_3 - B^t)\Z^3, \, \, \Z^3/
{\begin{bmatrix}
&     & 1 \\
& I_3 - B^t & \vdots \\
&     & 1
\end{bmatrix} \Z^{4})
}
& = (\Z / 2\Z, \Z/2\Z)
\end{align*}
Hence the two Cuntz-Krieger algebras $\OCKA$ and $\OCK_B$ are not isomorphic.

\medskip

{\it Acknowledgments:}
%The authors would like to thank anonymous referees for their careful reading
%of the manuscript and their helpful comments and suggestions.
K. Matsumoto is supported by JSPS KAKENHI Grant Number 24K06775.
T. Sogabe is supported by JSPS KAKENHI Grant Number 24K16934.

\end{document}